\documentclass[12pt]{amsart}
\usepackage{amsmath}
\usepackage{hyperref}
\usepackage{amssymb}
\usepackage{latexsym}
\usepackage{amscd}
\usepackage{graphicx} %We can use any other package if necessary
\usepackage{amsthm}
\usepackage{mathrsfs}
\usepackage{xypic}
\usepackage{bm}

\newdimen\AAdi%
\newbox\AAbo%
\def\AAk#1#2{\s_etbox\AAbo=\hbox{#2}\AAdi=\wd\AAbo\kern#1\AAdi{}}%
\def\AAr#1#2#3{\s_etbox\AAbo=\hbox{#2}\AAdi=\ht\AAbo\raise#1\AAdi\hbox{#3}}%
\font\tenmsb=msbm10 at 12pt \font\sevenmsb=msbm7 at 8pt
\font\fivemsb=msbm5 at 6pt
\newfam\msbfam
\textfont\msbfam=\tenmsb \scriptfont\msbfam=\sevenmsb
\scriptscriptfont\msbfam=\fivemsb
\def\Bbb#1{{\tenmsb\fam\msbfam#1}}
\newtheorem{thm}{Theorem}[section]
\newtheorem{lem}{Lemma}[section]

\newtheorem{rem}{Remark}[section]
\newtheorem{pro}{Proposition}[section]
\newtheorem{defi}{Definition}[section]

\newcommand{\ba}{\begin{array}}
\newcommand{\ea}{\end{array}}

\newcommand{\Section}[2]{\setcounter{equation}{0}
\allowdisplaybreaks
\section[#1]{#2}}

\def\pr {\noindent {\it Proof.} }

\def\n{\nabla}

\def\ir#1{\mathbb R^{#1}}
\def\hh#1{\Bbb H^{#1}}

\def\cc#1{\Bbb C^{#1}}
\def\f#1#2{\frac{#1}{#2}}

\def\grs#1#2{\bold G_{#1,#2}}

\def\pr{\frac {\partial}{\partial r}}

\def\pd#1#2{\frac {\partial #1}{\partial #2}}

\def\epw#1{\ep_1\w\cdots\w \ep_{#1}}
\def\td{\tilde}
 
\def\inner#1#2#3#4{(e_{#1},\ep_1)(e_{#2},\ep_2)(\nu_{#3},\ep_1)(\nu_{#4},\ep_2)}
\def\second#1#2{h_{\a,i#1}h_{\be,i#2}\lan e_{#1\a},A\ran\lan e_{#2\be},A\ran}

\def\a{\alpha}
\def\be{\beta}

\def\p#1{\partial #1}

\def\de{\delta}
\def\De{\Delta}
\def\e{\eta}
\def\ep{\varepsilon}
\def\eps{\epsilon}

\def\g{\gamma}

\def\la{\lambda}
\def\La{\Lambda}
\def\om{\omega}
\def\Om{\Omega}
\def\th{\theta}
\def\Th{\Theta}

\def\w{\wedge}

\def\Hess{\mbox{Hess}}
\def\R{\Bbb{R}}

\def\tr{\mbox{tr}}

\def\lan{\langle}
\def\ran{\rangle}
\def\ra{\rightarrow}

\def\aint#1{-\hskip -4.5mm\int_{#1}}


\begin{document}
\title
[Geometry of Grassmannian manifolds and Bernstein type theorems] {The geometry of Grassmannian manifolds and Bernstein type theorems for higher codimension}

\author
[J. Jost, Y. L. Xin and Ling Yang]{J. Jost, Y. L. Xin and Ling Yang}
\address{Max Planck Institute for Mathematics in the
Sciences, Inselstr. 22, 04103 Leipzig, Germany}
\address{Department of Mathematics and Computer Science, University of Leipzig,
04109 Leipzig, Germany.}
\email{jost@mis.mpg.de}
\address {Institute of Mathematics, Fudan University,
Shanghai 200433, China.} \email{ylxin@fudan.edu.cn}
\address{Institute of Mathematics, Fudan University,
Shanghai 200433, China.} \email{lingyang@fudan.edu.cn}
\thanks{The first author is supported by the ERC Advanced Grant
  FP7-267087, the second and third authors are partially supported by NSFC. They  are grateful to the Max Planck
Institute for Mathematics in the Sciences in Leipzig for its
hospitality and  continuous support. }

\begin{abstract}
We identify a region $\Bbb{W}_{\f{1}{3}}$ in a Grassmann manifold
$\grs{n}{m}$, not covered by a usual matrix coordinate chart, with
the following important property. For a complete $n-$submanifold in $\ir{n+m} \, (n\ge 3, m\ge2)$
with parallel mean curvature whose image under the Gauss map is
contained in a compact subset
$K\subset\Bbb{W}_{\f{1}{3}}\subset\grs{n}{m}$,
we can construct strongly subharmonic functions and derive a priori
estimates for the harmonic Gauss map. While we do not know yet how
close our region is to being optimal in this respect, it is
substantially larger than what could be achieved previously with other
methods. Consequently, this enables us to obtain substantially
 stronger Bernstein type theorems in higher
codimension than previously known.
% As for minimal submanifolds with the rank of the Gauss
% maps less or equal to $2$. The results are better and corresponding Gauss image would be larger.
\end{abstract}
\maketitle

\Section{Introduction}{Introduction}

The classical Bernstein theorem states that any complete minimal
graph $M^2$ in $\ir{3}$ has to be an affine plane. Equivalently, all
tangent planes are parallel to each other, and the Gauss map
$\g:M\to S^2$ is constant. This scheme extends to higher dimensions,
that is, minimal graphs in $\ir{n+1}$, although the result becomes
somewhat weaker for $n\ge 7$. In this work, we are interested in
such a scheme for higher codimensional minimal submanifolds of
Euclidean space.

We now make this more precise. Let $M^n$ be a complete $n$-dimensional submanifold in Euclidean space $\ir{n+m}$. The
Grassmann manifold $\grs{n}{m}$ is the target manifold of the Gauss
map $\g: M\to\grs{n}{m}$ that assigns to each point of $M$ the
direction of its tangent space in $\ir{n+m}$. By the Ruh-Vilms theorem
\cite{r-v}  the Gauss map is a harmonic map  into the Grassmann
manifold iff $M$ has parallel mean curvature. In particular, the Gauss
map is harmonic for minimal submanifolds of Euclidean space. When $m=1$, the
Grassmann manifold reduces to the sphere $S^n$, and investigating the
Gauss map has identified many conditions under which a minimal
submanifold has to be affine. In this scheme, the key is to show that
the Gauss map is constant.
The aim of the present  paper then is to obtain Bernstein type result for higher codimension by using the geometric properties  of the Grassmann manifolds $\grs{n}{m}$ and studying the corresponding harmonic Gauss maps into $\grs{n}{m}.$

This method was introduced by Hildebrandt-Jost-Widman \cite{h-j-w}. The distance function from a fixed point $P_0\in\grs{n}{m}$ in a geodesic ball of  radius
$\f{\sqrt{2}}{4}\pi$ and centered at $P_0$ is convex. Using this fact one obtains a strongly subharmonic function on $M$ by composing the distance function
with the Gauss map $\g$ under the assumption that the image of the
Gauss map is contained in a closed subset of the
geodesic convex ball. Suppressing some basic technical difficulties in
this introduction, an application of the maximum principle should then
yield that this subharmonic function is constant, and hence so then is
the Gauss map. Using this  local convex geometry of  Grassmann
manifolds and  advanced harmonic map regularity theory to overcome the
indicated technical difficulties,
Hildebrandt-Jost-Widman obtained Bernstein type results for higher
dimension and codimension. Somehow, however, while a geodesic ball of  radius
$\f{\sqrt{2}}{4}\pi$ is the largest convex {\it ball} for $m\ge 2$
(for $m=1$, we can take the ball of the radius
$\f{1}{2}\pi$, that is, a hemisphere), it is not the largest
convex {\it set}, and therefore the result of \cite{h-j-w} is not yet
the best possible, and there seems opportunity for improvement. In
\cite{j-x}, then,  the largest such geodesic convex set was found, and
stronger results were obtained. But, as we shall explore here, there
is still further opportunity for improvement, via a deeper understanding of
the convex geometry of Grassmann manifolds.

Since the geometry of general Grassmann manifolds is not as easy to
visualize as the one for the special case $m=1$ where the Grassmannian
reduces to the sphere $S^{n}$, let us briefly discuss the situation
for the latter. In general, the domain of a strictly convex function
cannot contain any closed geodesic, and therefore, a closed hemipshere
in $S^{n}$ cannot support a strictly convex function. Thus, no ball
with a radius $\ge \frac{\pi}{2}$ can support a strictly convex
function. Nevertheless, there exist larger open sets that contain an
open hemisphere, but still do not contain any closed geodesic. In
fact, on $S^2$, we can take the complement $U$ of a semicircle. In
\cite{j-x-y}, we have shown that any compact subset of $U$ supports a
strictly convex function. (It is, however, not true that this convex
supporting property holds for any open subset of $S^{n}$ containing
no closed geodesic.) First of all, this inspires us to look for such
convex supporting sets in general Grassmann manifolds. Secondly, the
reason why we need convex functions is that the postcomposition of a
harmonic map (in our case, the Gauss map of a minimal submanifold)
with a convex function is subharmonic so that we can apply the maximum
principle. When suitable technical conditions are met (their
verification is part of our technical achievements), this tells us
that the composition, being subharmonic, is constant. When the
function is nontrivial, we can then conclude that the Gauss map itself
is constant. Bingo -- the minimal submanifold is affine! Now, however,
it might happen that the postcomposition of the Gauss map with a
function $f$ turns out to be subharmonic without $f$ being
convex. Here, the more subtle geometry of the Grassmann manifolds for
$m\ge 2$ enters. An intuition about the geometry comes from the
following observation. We can rotate the plane $e_1\wedge e_2$ in
$\mathbb{R}^4$ either into the plane $e_1\wedge e_3$ or into the plane
$e_3\wedge e_4$. Each of these rotations yields a geodesic arc in
$\grs{2}{2}$. However, these two geodesics stop diverging from their neighboring geodesics at
different distances ($\f{\sqrt{2}}{4}\pi$ vs. $\frac{\pi}{2}$)
(equivalently, the first conjugate point is reached after a distance
of $\f{\sqrt{2}}{2}\pi$ vs. $\pi$), and therefore, a maximal geodesically convex set
is larger than a ball of radius $\f{\sqrt{2}}{4}\pi$. While this
geometric intuition will be important for the present paper, nevertheless,
due to more subtle features of Grassmannian geometry, we can construct
suitable functions that are no longer convex, but still give the
required subharmonicity on even larger sets than such maximal
geodesically convex sets. The function that we use is defined in a geometric manner from the Pl\"ucker embedding, or equivalently, in terms of Jordan angles. It is designed to take into account the phenomenon just described, namely that rotations of subspaces of Euclidean space are geometrically different depending on how many independent normal directions are involved. Again, this is a feature particular for
$m\ge 2$, not yet apparent on a sphere.  In
fact, and this is a main point of this paper, we shall find such
functions $f$ here, defined on larger subsets of $\grs{n}{m}$ than in
previous work (as we shall explain in more detail below, these subsets
are larger and more natural than the ones identified in our previous
work \cite{j-x-y2}). Therefore, we can obtain a larger possible range of the
Gauss map of a minimal submanifold that still implies that it is
constant. Thus, we obtain stronger Bernstein type results than
previously known. Actually, our approach, being quantitative, also yields estimates for the second fundamental form of minimal submanifolds under suitable geometric conditions.

Of course, there are limits how far one can push Bernstein
theorems in higher codimension. For codimension 1, we have Moser's
theorem \cite{m} that any entire minimal graph with bounded slope is affine
linear (for $n\le 7$, the bounded slope condition is not needed --
this is Simons' theorem \cite{Si}). In higher codimensions, the
situation is not so good. After all, there is the important example of Lawson-Osserman
\cite{l-o} of a nontrivial minimal graph (to be analyzed in detail in the Appendix of the present
paper) of bounded slope that sets a limit for how far one can go. The ultimate aim
then is to narrow the gap between the range of Bernstein theorems and
the counterexamples as
far as possible. Our paper is a step in this direction. We do not know
yet whether one can still go further, or whether there exist even more
striking examples than the one of \cite{l-o}.

One more remark: Most of what we do in this paper holds for
submanifolds of parallel mean curvature, and not only for those of
vanishing mean curvature, the minimal ones. Since, however, according
to mathematical tradition, the minimal submanifolds are the most
interesting ones, in this introduction, we mostly restrict ourselves to
discuss those minimal ones in place of the more general ones of parallel
mean curvature.

Let us now describe the results in more precise terms. The Grassmann
manifold $\grs{n}{m}$ can be imbedded into  Euclidean space by the
Pl\"ucker imbedding. This simply means that we consider an oriented
$n$-plane in $\mathbb{R}^{n+m}$ as an element of
$\La^n(\mathbb{R}^{n+m})$. In this space, we also have a scalar
product. We then can introduce suitable functions, called
$w$-functions and $v$-functions, on $\grs{n}{m}$. The $w$-function
is simply given by the above scalar product with some fixed
reference $n$-plane, and $v$ is the inverse of $w$. By
precomposition with the Gauss map, we then get corresponding
functions on $M$. The region of convexity for the $v$-function is
the same as for the corresponding distance function \cite{x-y1}, but
in our previous work \cite{j-x-y2} we have already found that there
is a larger region in $\grs{n}{m}$ on which the $v$-function is no
longer  convex, but where its composition with the Gauss map is
still  strongly subharmonic by  rather delicate estimates. We then
obtain the following theorem

\begin{thm}
Let $z^\a=f^\a(x^1,\cdots,x^n),\ \a=1,\cdots,m$, be smooth
functions defined everywhere in $\R^n$ ($n\geq 3,m\geq 2$).
Suppose their graph $M=(x,f(x))$ is a submanifold with parallel
mean curvature in $\R^{n+m}$. Suppose that there exists a number
$\be_0<3$ such that
\begin{equation}
\De_f=\Big[\det\Big(\de_{ij}+\sum_\a \f{\p f^\a}{\p x^i}\f{\p
f^\a}{\p x^j}\Big)\Big]^{\f{1}{2}}\leq \be_0.\label{be2}
\end{equation}
Then $f^1,\cdots,f^m$ has to be affine linear, i.e., it  represents an
affine $n$-plane.
\end{thm}

In our previous work (see \cite{j-x, j-x-y2, x2, x-y1}) the image under the Gauss map for a submanifold in Euclidean space
is contained in an $(n\times m)-$matrix chart
for a Grassmann manifold. There is, however, still further room for
improvement, and here, we consider a more general situation. Recall  that
a Grassmann manifold can be viewed as a minimal submanifold in the Euclidean
sphere via the Pl\"ucker imbedding. This leads us to employ a
technique from our previous work \cite{j-x-y} for the case of
codimension $m=1$, where the convex geometry
of the Euclidean sphere has been thoroughly investigated.

We shall introduce new notions of \textit{S-orthogonality} and of
\textit{S-maps} on a Grassmann manifold. Let $P, Q\in \grs{n}{m}$ be
$S-$orthogonal to each other. This means that their intersection is
of dimension $n-1$ and one is obtained from the other by rotating a
single tangent vector into a normal direction by an angle of
$\frac{\pi}{2}$. Using the $w$-function we define an $S$-map
$\mathscr{S}: \grs{n}{m}\to \overline{\Bbb{D}}$ relative to $P$ and
$Q$. This is a map onto the closed unit disk.

Define
\begin{equation}
\Bbb{W}_c:=\mathscr{S}^{-1}\Big(\overline{\Bbb{D}}\backslash \big(\overline{\Bbb{D}}_c\cup \big\{(a,0):a\leq 0\big\}\big)\Big),
\end{equation}
with $\overline{\Bbb{D}}_c=\big\{(x_1,x_2)\in \R^2: x_1^2+x_2^2\leq c^2\big\}$.

Now, for our method to apply, the image of our Gauss map can be any
compact subset of $\Bbb{W}_{\f{1}{3}}.$ This is a global region in
$\grs{n}{m}$ that is not contained in any matrix coordinate
chart. More precisely, even when the Gauss image is somewhat larger, we still can find subharmonic functions
on our submanifold by combining  some tricks in our previous work in
\cite{j-x-y} and \cite{j-x-y2}. In this way, we obtain the following
Bernstein type theorem.

\begin{thm}\label{1.2}
Let $z^\a=f^\a(x^1,\cdots,x^n),\ \a=1,\cdots,m$, be smooth functions
defined everywhere in $\R^n$ ($n\geq 3,m\geq 2$), such that their graph $M=(x,f(x))$ is a
submanifold with parallel mean curvature in $\R^{n+m}$. Suppose that
there exist $\be_0<+\infty$ and $\be_1<3$, such that
\begin{equation}
\De_f:=\Big[\det\Big(\de_{ij}+\sum_\a \f{\p f^\a}{\p x^i}\f{\p f^\a}{\p x^j}\Big)\Big]^{\f{1}{2}}\leq \be_0.\label{be2}
\end{equation}
and for certain $\a$ and $i$
\begin{equation}\label{slope}
\De_f\leq \be_1\Big(1+\big(\f{\p f^\a}{\p x^i}\big)^2\Big)^{\f{1}{2}}.
\end{equation}
Then $f^1,\cdots,f^m$ has to be affine linear (representing an affine $n$-plane).
\end{thm}

In fact, we can prove more general results, see Theorem \ref{Ber2} in
\S 6, but the preceding statement perhaps best highlights the main
achievement of this paper.

In the Appendix we see that for the Lawson-Osserman cone the
$v-$function is the constant $9$. The Jordan angles between the Gauss
image of  this cone and the
coordinate $n-$ plane are constants. The largest value for a Jordan
angle permitted by
the refined term $(1+\pd{f^2}{x^1})^{\f{1}{2}}$ in the above Theorem
is $\sqrt{6}$.
Theorem \ref{1.2}  is a substantial improvement of  the previous
results which only could reach smaller values.

The paper is organized as follows. We will describe in \S 2 that the basic geometry of Grassmann manifolds and define \textit{S-orthogonality} and
\textit{S-maps} and show their properties.  \S 3 will be devoted to the computation of $\De\log w$. For the general case the results come from \cite{x-y1}.
We also compute it for  minimal submanifolds with rank of the Gauss map less than or equal $2$. We obtain a formula more general than that in \cite{fc}. In \S 4
we will construct subharmonic functions $F$ on our submanifolds by using the techniques in \cite{j-x-y}. The level sets of $F$ coincide
with those of $\log w$. We first prove an important transition lemma (Lemma \ref{level0}). Then, we can estimate $\Hess\, F$ in terms of
$\Hess\,\log w$, which is already computed in the previous section. Once we have subharmonic functions, the extrinsic rigidity  results
for compact minimal submanifolds in the sphere (Theorem \ref{Ber1}) follows immediately. Using strongly subharmonic functions we can also study  complete
submanifolds in Euclidean space by using harmonic map regularity
theory,  as Hildebrandt-Jost-Widman in \cite{h-j-w}. In \S 5 we will
obtain curvature estimates (Lemma \ref{cur} ). In \S 6 the iteration
method will be used and   the quantitatively controlled Gauss image
shrinking Lemmas \ref{sh1} and  \ref{sh2} will be obtained. Finally,
in this section we prove our main results,  Theorem \ref{Ber2} and
Theorem \ref{Ber3} of which Theorem \ref{1.2} is a direct corollary. Finally, in an Appendix, we provide computations
for the Lawson-Osserman cone and related coassociated $4-$manifolds in $\ir{7}.$

\bigskip\bigskip
\Section{Geometry of Grassmann manifolds}{Geometry of Grassmann manifolds}\label{s1}

Let $\grs{n}{m}$ be the Grassmann manifold consisting of  the oriented linear $n-$subspaces in $(n+m)-$Euclidean space $\R^{n+m}.$
The canonical Riemannian structure on $\grs{n}{m}$ makes it a natural
generalization of the Euclidean sphere.    $\grs{n}{m}=SO(n+m)/SO(n)\times SO(m)$ is an
irreducible symmetric space of compact
type.

For every $P\in \grs{n}{m}$, we choose an oriented basis $\{u_1,\cdots,u_n\}$ of $P$, and
let
\begin{equation}
\psi(P):=u_1\w\cdots\w u_n\in \La^n(\R^{n+m}).
\end{equation}
A different basis for $P$ shall give a different exterior product,
but the two products differ only by a positive scalar; $\psi(P)$ is called the \textit{Pl\"ucker coordinate} of $P$, which is
a homogenous coordinate.

Via the Pl\"ucker embedding, $\grs{n}{m}$ can be viewed as a submanifold of some  Euclidean space $\ir{N} (N=C_ {n+m}^n)$. The restriction of the
Euclidean inner product  is denoted by $w:\grs{n}{m}\times \grs{n}{m}\ra \R$
\begin{equation}
w(P,Q)=\f{\lan \psi(P),\psi(Q)\ran}{\lan \psi(P),\psi(P)\ran^{\f{1}{2}}\lan \psi(Q),\psi(Q)\ran^{\f{1}{2}}}.
\end{equation}
If $\{e_1,\cdots,e_n\}$ is an oriented orthonormal basis of $P$ and $\{f_1,\cdots,f_n\}$ is an oriented
orthonormal basis of $Q$, then
$$w(P,Q)=\lan e_1\w\cdots\w e_n,f_1\w\cdots\w f_n\ran=\det W$$
with the \textit{W-matrix} $W=\big(\lan e_i,f_j\ran\big)$. It is well-known that
$$W^T W=O^T \La O$$
with  an orthogonal matrix $O$ and
$$\La=\left(\begin{array}{ccc}
            \mu_1^2 &   &  \\
                    & \ddots &  \\
                    &        & \mu_n^2
            \end{array}\right).$$
Here each $0\leq \mu_i^2\leq 1$. Putting $p:=\min\{m,n\}$, then
at most $p$ elements in $\{\mu_1^2,\cdots, \mu_n^2\}$ are not
equal to $1$. Without loss of generality,  we can assume
$\mu_i^2=1$ whenever $i>p$. We also note that the $\mu_i^2$ can
be expressed as
\begin{equation}\label{di1a}
\mu_i^2=\frac{1}{1+\la_i^2} \end{equation}
with $\la_i\in [0,+\infty)$.

The \textit{Jordan angles} between $P$ and $Q$ are critical values of the angle $\th$ between a nonzero vector
$u$ in $P$ and its orthogonal projection $u^*$ in $Q$ as $u$ runs through $P$. Let $\th_i$
be a Jordan angle between $P$ and $Q$ determined by a unit vector $e_i$ and its projection $e_i^*$ in $Q$, we call
$e_i$ an \textit{angle direction} of $P$ relative to $Q$, and the 2-plane spanned by $e_i$ and $e_i^*$ an \textit{angle 2-plane}
between $P$ and $Q$ (see \cite{w}). A direct calculation shows there are
$n$ Jordan angles $\th_1,\cdots,\th_n$, with $\th_{p+1}=\cdots=\th_n=0$ and
$$\th_i=\arccos(\mu_i)\qquad 1\leq i\leq p.$$
Thus
\begin{equation}\label{w1}
|w|=\big(\det(W^T W)\big)^{\f{1}{2}}=\det(\La)^{\f{1}{2}}=\prod_{i=1}^n \cos\th_i
\end{equation}
and (\ref{di1a}) becomes
\begin{equation}\label{di2}
\la_i=\tan\th_i. \end{equation}

If $w(P,Q)>0$, arrange all the Jordan angles between $P$ and $Q$ as
$$\f{\pi}{2}>\th_1\geq \th_2\geq \cdots\geq \th_r>\th_{r+1}=\cdots=\th_n=0$$
with $0\leq r\leq p$. (If $\th_i=\f{\pi}{2}$ for some $i$, then (\ref{w1}) implies $w(P,Q)=0$, which contradicts
 $w(P,Q)>0$.) Then one can find an orthonormal basis $\{e_1,\cdots,e_{n+m}\}$ of $\R^{n+m}$, such that
$P$ is spanned by $\{e_1,\cdots,e_n\}$, which are angle directions of $P$ relative to $Q$, and $\{e_i\w e_{n+i}:1\leq i\leq r\}$
are angle 2-planes between $P$ and $Q$. Denote
\begin{equation}
f_i:=\left\{\begin{array}{cc}
\cos\th_i e_i+\sin\th_i e_{n+i} & 1\leq i\leq r,\\
e_i & r+1\leq i\leq n.
\end{array}\right.
\end{equation}
and
\begin{equation}
f_{n+\a}:=\left\{\begin{array}{cc}
-\sin\th_\a e_\a+cos\th_\a e_{n+\a} & 1\leq \a\leq r,\\
e_{n+\a} & r+1\leq \a\leq m.
\end{array}\right.
\end{equation}
then $\{f_1,\cdots,f_{n+m}\}$ is also an orthonormal basis of $\R^{n+m}$, and $f_1\w\cdots\w f_n$ is a Pl\"ucker coordinate of $Q$.

The distance between $P$ and $Q$ is defined by
\begin{equation}\label{di}
d(P, Q)=\sqrt{\sum\th_i^2}. \end{equation}
It is a natural generalization of the canonical distance of Euclidean spheres.

Now we fix $P_0\in \grs{n}{m}.$ We represent it by $n$ vectors
$\eps_i$, which are complemented by $m$ vectors $ \eps_{n+\a}$.
Denote
\begin{equation}
\Bbb{U}:=\{P\in \grs{n}{m}: w(P,P_0)>0\}.
\end{equation}
We can span an arbitrary $P\in \Bbb{U}$ by $n$ vectors $e_i$:
\begin{equation}
e_i=\eps_i+Z_{i \a}\eps_{n+\a}.
\end{equation}
Here $Z=(Z_{ i \a})$ could be regarded as the $(n\times m)$-matrix coordinate of $P$. The canonical Riemannian metric in $\Bbb{U}$ can be described as
\begin{equation}\label{m1}ds^2 = tr (( I_n+ ZZ^T )^{-1} dZ (I_m + Z^TZ)^{-1} dZ^T
),\end{equation} where $ I_m $ (res. $ I_n $) denotes the $ (m\times m)
$-identity (res. $ n \times n $) matrix. It is shown that
(\ref{m1}) can be derived from (\ref{di}) in \cite{x}.

Here and in the sequel,  we use the
summation  convention and agree on the ranges of indices:
$$1\leq \a,\be\leq m,\; 1\leq  i,j,k\leq n.$$

We shall now introduce  the new concepts of \textit{S-orthogonal} and \textit{S-map}, which will play a crucial role for our  investigations.

\begin{pro}\label{orthogonal}
For any $P,Q\in \grs{n}{m}$, the following statements are equivalent:\newline
(a) $\dim (P\cap Q)=n-1$ and $w(P,Q)=0$;\newline
(b) $\dim (P+Q)=n+1$ and $w(P,Q)=0$, where $P+Q=\{u+v:u\in P, v\in Q\}$;\newline
(c) There exists a single Jordan angle between $P$ and $Q$ being
$\f{\pi}{2}$, and the other $(n-1)$ Jordan angles are all $0$;\newline
(d) $Q$ is the nearest point to $P$ in $\p \Bbb{U}$, where $\Bbb{U}$ is the $(n\times m)$-matrix coordinate chart centered at $P$;\newline
(e) There exists an orthonormal basis $\{e_i,e_{n+\a}\}$ of $\R^{n+m}$, such that the Pl\"ucker coordinates of $P$ and $Q$
are $e_1\w e_2\w\cdots\w e_n$ and $e_{n+1}\w e_2\w \cdots\w e_n$, respectively.

\end{pro}

\begin{proof}

(a)$\Longleftrightarrow$(b)  follows from  $\dim (P+Q)=\dim P+\dim Q-\dim (P\cap Q)$.

(a)$\Longleftrightarrow$(c) is an immediate corollary of the definition of Jordan angles and (\ref{w1}).

$\p \Bbb{U}$ consists of all $S\in \grs{n}{m}$ satisfying $w(P,S)=0$; it follows from $(\ref{w1})$ that at least 1
Jordan angle between $S$ and $P$ is $\f{\pi}{2}$; hence one can obtain
$$d(P,S)\geq \f{\pi}{2}\qquad \text{for all }S\in \p\Bbb{U}$$
from (\ref{di}), and the equality holds if and only if the other $(n-1)$ Jordan angles all vanish. Hence (c)$\Longleftrightarrow$(d).

(e)$\Longrightarrow$(a) is trivial. To prove (a)$\Longrightarrow$(e), it suffices to choose $\{e_2,\cdots,e_n\}$ as an orthonormal
basis of $P\cap Q$, and put $e_1$ (or $e_{n+1}$) to be the unit vector in $P$ (or $Q$) that is orthogonal to $P\cap Q$.

\end{proof}

\begin{defi}
Two points $P$ and $Q$ in a Grassmann manifold are called \textit{S-orthogonal},  if they satisfy one of the   properties in Proposition \ref{orthogonal}.
\end{defi}

\begin{rem}
There is only one Jordan angle between any two points in the sphere $S^n$. If it equals $\f{\pi}{2}$, they are orthogonal each other. The notion of  S-orthogonal is a natural generalization of  orthogonality in $S^n$.
\end{rem}

Let $P,Q\in \grs{n}{m}$ be S-orthogonal. By Proposition \ref{orthogonal}, there exists
an orthonormal basis $\{e_i,e_{n+\a}\}$ of $\R^{n+m}$, such that $\psi(P)=e_1\w e_2\w\cdots\w e_n$
and $\psi(Q)=e_{n+1}\w e_2\w \cdots\w e_n$. Let
 $\g: \R/(2\pi\Bbb{Z})\ra \grs{n}{m}$ be a closed curve, such that $\g(t)$ is spanned by
$\{\cos t\ e_1+\sin t\ e_{n+1}, e_2,\cdots, e_n\}$.
Then whenever $|t-s|<\f{\pi}{2}$, there exists one and only one Jordan angle between $\g(t)$ and $\g(s)$
being $|t-s|$, and the other Jordan angles are all $0$; hence by (\ref{di}),
$$d(\g(t),\g(s))=|t-s|\qquad \text{whenever }|t-s|<\f{\pi}{2}.$$
It implies $\g$
is the closed geodesic that is extended from the minimal geodesic between $P$ and $Q$.
In the following, we write $P_t=\g(t)$; in particular, $P_0=P$ and $P_{\f{\pi}{2}}=Q$.

The \textit{S-map}  $\mathscr{S}:\grs{n}{m}\ra \R^2$ is defined by
\begin{equation}
S\mapsto \big(w(S,P_0),w(S,P_{\f{\pi}{2}})\big).
\end{equation}
If there exists a nonzero vector in $S$  orthogonal to $P_0\cap P_{\f{\pi}{2}}$, then
$w(S,P_0)=w(S,P_{\f{\pi}{2}})=0$, and hence $\mathscr{S}(S)=0$. Otherwise, the orthogonal projection $p$
from $P_0\cap P_{\f{\pi}{2}}$ to $S$ has rank $n-1$, which enables us to find a unit vector
$f_1$ in $S$ which is orthogonal to the $(n-1)$-dimensional projective image $p(P_0\cap P_{\f{\pi}{2}})$.
Certainly, one can get
\begin{equation}\aligned\label{in5}
w(S,P_0)&=\lan f_1,e_1\ran w\big(p(P_0\cap P_{\f{\pi}{2}}),P_0\cap P_{\f{\pi}{2}}\big),\\
w(S,P_{\f{\pi}{2}})&=\lan f_{1},e_{n+1}\ran w\big(p(P_0\cap P_{\f{\pi}{2}}),P_0\cap P_{\f{\pi}{2}}\big).
\endaligned
\end{equation}
It follows that
$$w^2(S,P_0)+w^2(S,P_{\f{\pi}{2}})\leq \lan f_1,e_1\ran^2+\lan f_1,e_{n+1}\ran^2\leq 1$$
and the equality holds if and only if $p(P_0\cap P_{\f{\pi}{2}})=P_0\cap P_{\f{\pi}{2}}$ and $f_1\in \text{span}\{e_1,e_{n+1}\}$;
i.e. $S=P_t$ for some $t\in [-\pi,\pi)$. Hence $\mathscr{S}$ is a smooth map onto the closed unit 2-disk $\overline{\Bbb{D}}$. Although $\mathscr{S}$
is not one-to-one, any point on $\p \overline{\Bbb{D}}$ has one and only one preimage. For simplicity we write
\begin{equation}
x_1=w(\cdot,P_0),\qquad x_2=w(\cdot,P_{\f{\pi}{2}}).
\end{equation}
Then from (\ref{in5}) one can obtain
\begin{equation}\label{w3}\aligned
(\cos t\ x_1+\sin t\ x_2)(S)&=\cos t\ w(S,P_0)+\sin t\ w(S,P_{\f{\pi}{2}})\\
                        &=\lan f_1,\cos t\ e_1+\sin t\ e_{n+1}\ran w\big(p(P_0\cap P_{\f{\pi}{2}}),P_0\cap P_{\f{\pi}{2}}\big)\\
                        &=w(S,P_t).
                        \endaligned
\end{equation}

\bigskip\bigskip

\Section{Laplacian of $\log w$}{Laplacian of $\log w$}

Let $M^n\to \bar{M}^{n+m}$ be an isometric immersion with  second
fundamental form $B,$ which can be viewed as a cross-section of the
vector bundle Hom($\odot^2TM, NM$) over $M,$ where  $TM$ and $NM$
denote the tangent bundle and the normal bundle  along $M$,
respectively. The connection on $TM$ and $NM$ (denoted by $\n$) can
be induced naturally from the Levi-Civita connection on $\bar{M}$ (denoted by
$\bar{\n}$).
We define the mean curvature $H$ as the trace of the second
fundamental form. It is a normal vector field on $M$ in $\bar{M}$. If
$\n H=0$, we say $M$ has \textit{parallel mean curvature}; moreover
if $H$ vanishes on $M$ everywhere, it is called a \textit{minimal submanifold}.

The second fundamental form, the curvature tensor of the submanifold
(denoted by $R$), the
curvature tensor of the normal bundle (denoted by $R^N$) and that of the ambient
manifold (denoted by $\bar{R}$) satisfy the Gauss equations, the Codazzi equations and the
Ricci equations (see \cite{x} for details, for example). Here and in
the sequel,
$\{e_i\}$ denotes a local orthonormal tangent frame field and
$\{\nu_{\a}\}$ is a local orthonormal normal frame field of $M$;
$$h_{\a, ij}:=\lan \lan B(e_i,e_j),\nu_\a\ran$$
are the coefficients of the second fundamental form $B$ of $M$ in $\bar{M}$.

Now we consider a submanifold $M^n$   in
Euclidean space $\ir{n+m}$  with parallel mean curvature.

Let $0$ be the origin of $\R^{n+m}$, $SO(n+m)$ be the Lie group
consisting of all orthonormal frames $(0;e_i,\nu_{\a})$.
$TF=\big\{(x;e_1,\cdots,e_n):x\in M,e_i\in T_p M,\lan
e_i,e_j\ran=\de_{ij}\big\}$ be the principle bundle of orthonormal
tangent frames over $M$, and
$NF=\big\{(x;\nu_{1},\cdots,\nu_{m}):x\in M,\nu_{\a}\in N_x M\big\}$
be the principle bundle of orthonormal normal frames over $M$. Then
$\bar{\pi}: TF\oplus NF\ra M$ is the projection with fiber
$SO(n)\times SO(m)$.

The Gauss map $\g: M\ra \grs{n}{m}$ is defined by
$$\g(x)=T_x M\in \grs{n}{m}$$
via the parallel translation in $\R^{n+m}$ for every $p\in M$. Then the following  diagram commutes
$$\CD
 TF \oplus NF @>i>> SO (n+m)  \\
 @V{\bar\pi}VV     @VV{\pi}V \\
 M  @>{\g}>>  \grs{n}{m}
\endCD$$
where $i$ denotes the inclusion map and $\pi: SO(n+m)\ra \grs{n}{m}$ is defined by
$$(0;e_i,\nu_{\a})\mapsto e_1\w\cdots\w e_n.$$
From the above diagram we know that the energy density of the Gauss
map (see \cite{x} Chap.3, \S 3.1)
$$e(\g)=\f{1}{2}\left<\g_*e_i,\g_*e_i\right>=\f{1}{2}|B|^2.$$
Ruh-Vilms proved  that the mean curvature vector of $M$ is parallel
if and only if its Gauss map is a harmonic map \cite{r-v}.

We define
\begin{equation}
w:=w(\cdot,P_0)\circ \g.
\end{equation}
Let $\{\ep_i\}$ be an orthonormal basis of $P_0$, then $w=\lan e_1\w \cdots\w e_n,\ep_1\w\cdots\w \ep_n\ran$ at every point $x\in M$.
From the Codazzi equations, we shall now get basic formulas for the function $w$.

\begin{lem}\cite{fc}\cite{x1}
If $M$ is a submanifold in $\R^{n+m}$, then
\begin{equation}\label{dw}
\n_{e_i} w=h_{\a,ij}\lan e_{j\a},\ep_1\w\cdots\w \ep_n\ran
\end{equation}
with
\begin{equation}
e_{j\a}=e_1\w\cdots \w \nu_\a\w \cdots\w e_n
\end{equation}
that is obtained by replacing $e_j$ by $\nu_\a$ in $e_1\w\cdots\w e_n$. Moreover if $M$
has parallel mean curvature, then
\begin{equation}\label{La}
\De w=-|B|^2 w+\sum_i\sum_{\a\neq \be,j\neq k}h_{\a,ij} h_{\be,ik} \lan e_{j\a,k\be},\ep_1\w \cdots\w \ep_n\ran.
\end{equation}
with
\begin{equation}
e_{j\a,k\be}=e_1\w\cdots\w \nu_\a \w\cdots\w \nu_\be\w \cdots\w e_n
\end{equation}
that is obtained by replacing $e_j$ by $\nu_\a$ and $e_k$ by $\nu_\be$ in $e_1\w \cdots\w e_n$, respectively.

\end{lem}

Now we compute $\De \log w$ under the additional assumption that $H=0$
and the rank of the Gauss map $\g$ is at most 2.
Without loss of generality, one can assume $h_{ij}^\a=0$ whenever $i\geq 3$ or $j\geq 3$.
Then, we need to calculate $\lan e_{1\a,2\be},\epw{n}\ran$. The following identity shall play an important role.
\begin{lem}\label{pluck}
Fix $A=\epw{n}$, then for any distinct indices $\a,\be$,
\begin{equation}\label{wedge}
\lan e_1\w\cdots\w e_n,A\ran\lan e_{1\a,2\be},A\ran-\lan e_{1\a},A\ran\lan e_{2\be},A\ran+\lan e_{1\be},A\ran\lan e_{2\a},A\ran=0.
\end{equation}
\end{lem}

\begin{proof}

Let $Q\in \grs{n-2}{m+2}$ spanned by $\{e_i:3\leq i\leq n\}$. If there is a nonzero vector
in $Q$ which is orthogonal to the $n$-dimensional space $P_0$ spanned by $\{\ep_i\}$, then by the definition
of the inner product on $\La^n(\R^{n+m})$, all the terms on the left hand side of (\ref{wedge}) equal $0$ and (\ref{wedge})
trivially holds true.

Otherwise the orthogonal projection $p: Q\ra P_0$ has rank $n-2$. Without loss of generality we assume $p(Q)$ is spanned
by $\{\ep_i:3\leq i\leq n\}$; the Jordan angles between $Q$ and $p(Q)$ are denoted by $\th_3,\cdots,\th_n$; and $\th_i$ is determined
by $e_i$ and $\ep_i$. Hence $\lan e_i,\ep_j\ran=\cos\th_i\de_{ij}$ whenever $i\geq 3$, and moreover
\begin{equation}\label{in1}\aligned
\lan e_1\w\cdots\w e_n,A\ran&=\left|\begin{array}{ccccc}
(e_1,\ep_1) & (e_1,\ep_2) & & & \\
(e_2,\ep_1) & (e_2,\ep_2) & &\multicolumn{1}{c}{\raisebox{-0.5ex}[0pt]{\Huge *}} &\\
            &             & \cos\th_3 & &\\
            & \multicolumn{1}{l}{\raisebox{-0.5ex}[0pt]{\Huge 0}} & & \ddots &\\
            &             &           & &\cos\th_n
\end{array}
\right|\\
&=\lan e_1\w e_2,\ep_1\w \ep_2\ran\prod_{i=3}^n \cos\th_i
\endaligned
\end{equation}
Similarly,
\begin{equation}\label{in3}
\lan e_{1\a,2\be},A\ran=\lan \nu_\a\w \nu_\be,\ep_1\w \ep_2\ran\prod_{i=3}^n \cos\th_i
\end{equation}
and
\begin{eqnarray}
\lan e_{1\g},A\ran&=&\lan \nu_\g\w e_2,\ep_1\w\ep_2\ran\prod_{i=3}^n \cos\th_i\\
\lan e_{2\g},A\ran&=&\lan e_1\w\nu_\g ,\ep_1\w\ep_2\ran\prod_{i=3}^n \cos\th_i
\end{eqnarray}
for $\g=\a$ or $\be$.
A direct calculation shows
\begin{equation}\label{in2}\aligned
&\lan e_1\w e_2,\ep_1\w\ep_2\ran\lan \nu_\a\w\nu_\be,\ep_1\w\ep_2\ran-\lan \nu_\a\w e_2,\ep_1\w \ep_2\ran\lan e_1\w \nu_\be,\ep_1\w \ep_2\ran\\
&+\lan \nu_\be\w e_2,\ep_1\w \ep_2\ran\lan e_1\w \nu_\a,\ep_1\w \ep_2\ran\\
=&+\inner{1}{2}{\a}{\be}+\inner{2}{1}{\be}{\a}\\
&-\inner{1}{2}{\be}{\a}-\inner{2}{1}{\a}{\be}\\
&-\inner{1}{2}{\a}{\be}-\inner{2}{1}{\be}{\a}\\
&+(e_1,\ep_2)(e_2,\ep_2)(\nu_\a,\ep_1)(\nu_\be,\ep_1)+(e_1,\ep_1)(e_2,\ep_1)(\nu_\a,\ep_2)(\nu_\be,\ep_2)\\
&+\inner{1}{2}{\be}{\a}+\inner{2}{1}{\a}{\be}\\
&-(e_1,\ep_2)(e_2,\ep_2)(\nu_\a,\ep_1)(\nu_\be,\ep_1)-(e_1,\ep_1)(e_2,\ep_1)(\nu_\a,\ep_2)(\nu_\be,\ep_2)\\
=&0.
\endaligned
\end{equation}
From (\ref{in1})-(\ref{in2}), (\ref{wedge}) immediately follows.

\end{proof}

$H=0$ implies $0=h_{\a,ii}=h_{\a,11}+h_{\a,22}$ for every $1\leq \a\leq m$, thus
\begin{equation}\label{eq}\aligned
h_{\a,i1}h_{\be,i2}=&h_{\a,11}h_{\be,12}+h_{\a,21}h_{\be,22}\\
                   =&-h_{\a,22}h_{\be,12}-h_{\a,21}h_{\be,11}=-h_{\a,i2}h_{\be,i1}
\endaligned
\end{equation}
and in particular $h_{\a,i1}h_{\a,i2}=0$. Therefore (\ref{La}) can be rewritten as
\begin{equation}\aligned
\De w&=-|B|^2 w+\sum_{\a\neq \be}h_{\a,i1}h_{\be,i2}\lan e_{1\a,2\be},A\ran+\sum_{\a\neq \be}h_{\a,i2}h_{\be,i1}\lan e_{1\be,2\a},A\ran\\
     &=-|B|^2 w+2\sum_{\a, \be}h_{\a,i1}h_{\be,i2}\lan e_{1\a,2\be},A\ran.
\endaligned
\end{equation}
By (\ref{dw}),
\begin{equation}
|\n w|^2=\sum_i |\n_{e_i}w|^2=\second{j}{k}.
\end{equation}
With the aid of (\ref{wedge}) and (\ref{eq}) one can get
\begin{equation}\aligned
w\De w-|\n w|^2=&-|B|^2 w^2+2h_{\a,i1}h_{\be,i2}\lan e_1\w\cdots\w e_n,A\ran\lan e_{1\a,2\be},A\ran-\second{j}{k}\\
               =&-|B|^2 w^2+2h_{\a,i1}h_{\be,i2}\Big(\lan e_{1\a},A\ran\lan e_{2\be},A\ran-\lan e_{1\be},A\ran\lan e_{2\a},A\ran\Big)\\
                &-\second{j}{k}\\
               =&-|B|^2w^2+2\second{1}{2}+2\second{2}{1}\\
                &-\second{1}{1}-\second{2}{2}\\
                &-\second{1}{2}-\second{2}{1}\\
               =&-|B|^2w^2-\sum_i\Big(\sum_\a\big(h_{\a,i1}\lan e_{1\a},A\ran-h_{\a,i2}\lan e_{2\a},A\ran\big)\Big)^2.
               \endaligned
\end{equation}
Noting that $\De \log w=w^{-2}(w\De w-|\n w|^2)$, we arrive at

\begin{pro}\label{rank}
Let $M$ be a minimal submanifold of $\R^{n+m}$, with the rank of Gauss
map at most $2$. If $w>0$ at a point, then locally
\begin{equation}
\De\log w\leq -|B|^2.
\end{equation}

\end{pro}

\begin{rem}
If $N$ is a 2-dimensional minimal submanifold in $S^{2+m}$, then the
cone $CM$ over $M$ is a 3-dimensional minimal submanifold
in $\R^{3+m}$. The rank of the Gauss map from $CM$ is no more than 2,
as will be shown in the next section. Hence Proposition \ref{rank}
is a generalization of Proposition 2.2 in \cite{fc}. We note that Lemma \ref{pluck} is also a generalization of Lemma 2.1 in \cite{fc}.
\end{rem}

Finally, we consider the general case. For any $x\in M$ satisfying $w(x)>0$, then $w(P,P_0)>0$ with $P=\g(x)$.
Let $\th_1,\cdots,\th_n$ be the Jordan angles between $P$ and $P_0$ which are arranged as
\begin{equation}
\f{\pi}{2}>\th_1\geq \cdots\geq \th_r>\th_{r+1}=\cdots=\th_n=0
\end{equation}
with $0\leq r\leq \min\{n,m\}$. As shown in Section \ref{s1}, one can find an orthonormal basis $\{\ep_i,\ep_{n+\a}\}$
of $\R^{n+m}$, such that $\psi(P_0)=\ep_1\w\cdots\w \ep_n$; if we denote
\begin{equation}\label{tangent}
e_i:=\left\{\begin{array}{cc}
\cos\th_i\ep_i+\sin\th_i\ep_{n+i} & 1\leq i\leq r\\
\ep_i & r+1\leq i\leq n
\end{array}\right.
\end{equation}
and
\begin{equation}\label{normal}
\nu_\a:=\left\{\begin{array}{cc}
-\sin\th_\a \ep_\a+\cos\th_\a \ep_{n+\a} & 1\leq \a\leq r\\
\ep_{n+\a} & r+1\leq \a\leq m
\end{array}\right.
\end{equation}
then $\{e_i,\nu_\a\}$ is also an orthonormal basis of $\R^{n+m}$ and $\psi(P)=e_1\w \cdots\w e_n$.

Now we use the notation
\begin{equation}
v:=w^{-1}
\end{equation}
as in \cite{x-y1} and \cite{j-x-y2}; the function $v$ is well-defined in a neighborhood of $x$. By
a direct calculation based on metric forms in terms of matrix coordinates (\ref{m1}), one can derive
\begin{equation}
\De v=v|B|^2+v\sum_{i,j}2\la_j^2h_{j,ij}^2+v\sum_i\sum_{j\neq k}\la_j\la_k(h_{j,ij}h_{k,ik}+h_{k,ij}h_{j,ik})
\end{equation}
as in \cite{j-x-y2}; where $\la_j=\tan\th_j$ and $h_{\a,ij}=\lan B(e_i,e_j),\nu_\a\ran$. Grouping
the terms  according to different types of the indices of the coefficients of the second fundamental
form, one can proceed as in the proof of \cite{j-x-y2} Proposition 3.1 to get
\begin{equation}
v^{-1}\De v\geq C_1|B|^2\qquad \text{whenever }v\leq \be<3.
\end{equation}
 Noting that $v=\exp(-\log w)$, one can immediately get

\begin{pro}\label{subhar3}
If $M$ is a submanifold in $\R^{n+m}$ with parallel mean curvature, once $w\geq \f{1}{3}+\de$
with a positive number $\de$, then
\begin{equation}\label{La4}
\De \log w-|\n \log w|^2\leq -C_1|B|^2\qquad \text{at }x
\end{equation}
with a positive constant $C_1$ depending only on $\de$.
\end{pro}

\bigskip\bigskip

\Section{Subharmonic  functions and extrinsic rigidity problem}{Subharmonic  functions and extrinsic rigidity  problem}\label{s3}

We already defined the \textit{S-map} $\mathscr{S}:\grs{n}{m}\to \overline{\Bbb{D}}$ in \S\ref{s1}. For every $(x_1,x_2)\in \overline{\Bbb{D}}\backslash \big\{(a,0):a\leq 0\big\}$
that is obtained by deleting a radius from the disk,
there exist unique $r\in (0,1]$ and $\th\in (-\pi,\pi)$, such that
$$(x_1,x_2)=(r\cos\th,r\sin\th).$$
Here $r$ and $\th$ can be seen as smooth functions on $\mathscr{S}^{-1}\Big(\overline{\Bbb{D}}\backslash \big\{(a,0):a\leq 0\big\}\Big)$.
For any $c\in [0,1)$, we put
\begin{equation}
\Bbb{W}_c:=\mathscr{S}^{-1}\Big(\overline{\Bbb{D}}\backslash \big(\overline{\Bbb{D}}_c\cup \big\{(a,0):a\leq 0\big\}\big)\Big),
\end{equation}
with $\overline{\Bbb{D}}_c=\big\{(x_1,x_2)\in \R^2: x_1^2+x_2^2\leq
c^2\big\}$, then the function $r$ takes values between $c$ and $1$ on $\Bbb{W}_c$.

\begin{rem}
If $m=1$, then $\grs{n}{m}=S^n$ and $\Bbb{W}_0$ is the complement of a
half-equator $\overline{S}_+^{n-1}$; in particular,
if $n=2$, then $\Bbb{W}_0$ can be obtained by deleting  half of a great circle between 2 antipodal points from
$S^2$. It has been proved in \cite{j-x-y} that $\Bbb{W}_0$ is a maximal convex-supporting set in $S^n$.
\end{rem}

\begin{rem}
Obviously, $P_t, P_{t+\pi}\in \Bbb{W}_c$ consist of an antipodal pair in $\grs{n}{m}$ for each $t\in (-\f{\pi}{2},0)$, which implies
$w(P_t,S)=-w(P_{t+\pi},S)$ for arbitrary $S\in \grs{n}{m}$. Hence $\Bbb{W}_c$ cannot lie in an $(n\times m)$-matrix coordinate
chart centered at any point of $\grs{n}{m}$. Furthermore, they are
conjugate to each other. Since the Grasmann manifolds are simply
connected symmetric spaces of compact type, the cut point of $P_t$ along a geodesic coincides with  the first conjugate point along
the geodesic. Therefore, $\Bbb{W}_c$ contains a cut point of $P_t$ at least.
\end{rem}

On  $\Bbb{W}_c$ there are smooth functions $r$ and $\th$. By the previous work \cite{j-x-y} our concerned function $F$ was constructed from
$r$ and $\th$ and we need to estimate $\Hess\ F$. Hopefully, the level set of $F=t$ coincides with level set of $w(\cdot, P_t)$. For estimating
$\Hess\ F$ in terms of $\Hess\ w(\cdot, P_t)$ we need the following
\textit{ transition Lemma}, which is crucial in our construction, but
may also be useful in other contexts.

\begin{lem}\label{level0}

Let $(N,h)$ be a Riemannian manifold, $F$ be a smooth function on $N$, and $\{H(\cdot,t):t\in \R\}$ be a smooth family
of smooth functions on $N$. Assume that there are smooth real functions $t(s)$ and $l(s)$; for all $s\in \R$, the level set
$F_s:=\{p\in N: F(p)=s\}$ coincides with $\{p\in N: H\big(p,t(s)\big)=l(s)\}$;
$\n F$ and $\n H\big(\cdot,t(s)\big)$ are nonzero normal vector fields on $F_s$, pointing in the same
direction. Then for arbitrary $\ep>0$, there exists
a smooth nonpositive function $\la$ on $N$, depending on $\ep$, such that
\begin{equation}\label{level}
\Hess\ F\geq \f{|\n F|}{\big|\n H\big(\cdot,t(s)\big)\big|}\Hess\ H\big(\cdot,t(s)\big)-\ep h+\la d F\otimes d F
\end{equation}
on $N$.

\end{lem}

\begin{proof}
For any $s\in \R$,
$F_s$ can be regarded
as a hypersurface of $N$.
Denote by $\n$ and $\n^s$ the Levi-Civita connection on $N$ and $F_s$, respectively. We put
$\nu=\f{\n F}{|\n F|}$, then $\nu$ is the unit normal vector field on $F_s$; and denote by
$B$ the second fundamental form of $F_s$ in $N$.

For every $X\in T(F_s)$,
\begin{equation}\aligned
\Hess\ F(X,X)&=\n_X\n_X F-(\n_X X)F\\
                &=-(\n^s_X X)F-\lan B(X,X),\nu\ran \lan \nu,\n F\ran\\
                &=-|\n F|\lan B(X,X),\nu\ran.
                \endaligned
\end{equation}
Similarly,
\begin{equation}
\Hess\ H\big(\cdot,t(s)\big)(X,X)=-\big|\n H\big(\cdot,t(s)\big)\big|\lan B(X,X),\nu\ran.
\end{equation}
Hence
\begin{equation}
\Big(\Hess\ F-\f{|\n F|}{\big|\n H\big(\cdot,t(s)\big)\big|}\Hess\ H\big(\cdot,t(s)\big)\Big)(X,X)=0\qquad \text{for all }X\in T(F_s).
\end{equation}
It implies the existence of a smooth 1-form $\om$ on $N$, such that
\begin{equation}\label{level1}
\Hess\ F-\f{|\n F|}{\big|\n H\big(\cdot,t(s)\big)\big|}\Hess\ H\big(\cdot,t(s)\big)=\om\otimes dF+dF\otimes \om
\end{equation}
on $F_s$. Denote by $\om^*$ the associated vector field of $\om$ with respect to $h$, i.e.
$$\lan Y,\om^*\ran=\om(Y)\qquad \text{for all }Y\in TN.$$
Then
\begin{equation}\label{level2}\aligned
&(\om\otimes dF+dF\otimes \om)(Y,Y)\\
=& 2\om(Y)dF(Y)=2\lan \om^*,Y\ran dF(Y)\\
\geq& -2|\om^*||Y|dF(Y)\geq -\ep|Y|^2-\ep^{-1}|\om^*|^2dF(Y)^2.
\endaligned
\end{equation}
Let $\la=-\ep^{-1}|\om^*|^2$, then substituting (\ref{level2}) into (\ref{level1}) yields (\ref{level}).

\end{proof}

Given a compact subset $K$ of $\Bbb{W}_c$, there exists a positive number $\de$, such that $r\geq c+2\de$ on $K$. Hence
the following function
\begin{equation}\label{fun0}
F=\th-\arccos\left(\f{c+\de}{r}\right)
\end{equation}
is well-defined on $K$. Once $F(S)=t$, one can deduce that $\th-t=\arccos(\f{c+\de}{r})\in (0,\f{\pi}{2})$, and
\begin{equation}\label{fun}\aligned
c+\de&=r(\cos(\th-t))=r(\cos t\cos\th+\sin t\sin\th)\\
     &=(\cos t\ x_1+\sin t\ x_2)(S)=w(S,P_t).
\endaligned
\end{equation}
In other words, the level set $F_t$ of $F$ overlaps with the level set $\{S\in \grs{n}{m}:-\log w(S,P_t)=-\log(c+\de)\}$.

To apply Lemma \ref{level0}, we need gradient estimates for the functions $-\log w(\cdot,P_t)$ and $F$
along the common level set.

Let $\Bbb{U}$ be the matrix coordinate chart centered at $P_t$. Recall that $w(\cdot,P_t)$ has an expression
\begin{equation}
w=\Big[\tr(I_n+ZZ^T)\Big]^{-1}
\end{equation}
in terms of matrix coordinates (see \cite{x-y1}). It is easily seen
that both $w$ and the metric on $\Bbb{U}$ are invariant under
$SO(n)\times SO(m)$-actions. Hence for any $S\in \Bbb{U}$, without loss of generality one can assume
$$Z(S)=(\la_i\de_{i\a})=(\tan\th_i\de_{i\a})$$
with $\{\th_i\}$ being the Jordan angles between $S$ and $P_t$. As in  \cite{x-y1}, one can obtain
$$dw(\cdot,P_t)=-\sum_{1\leq i\leq p}\la_i w(\cdot,P_t)\om_{ii}\qquad \text{at }R$$
with $p=\min\{n,m\}$ and $\{\om_{i\a}\}$ is the dual basis of
$\big\{(1+\la_i^2)^{\f{1}{2}}(1+\la_\a^2)^{\f{1}{2}}E_{i\a}\big\}$, which is an orthogonal basis
of $T_S \grs{n}{m}$. (Here $E_{i\a}$ is the matrix with 1 in the intersection of row $i$ and column $\a$ and 0 otherwise.)
Therefore
\begin{equation}
\big|\n \log w(\cdot,P_t)\big|^2=\sum_{1\leq i\leq p}\la_i^2\geq p\big(w(\cdot,P_t)^{-\f{2}{p}}-1\big)
\end{equation}
and the equality holds if and only if
$\la_1^2=\cdots=\la_p^2=w(\cdot,P_t)^{-\f{2}{p}}-1$. In particular, on $F_t$,
\begin{equation}\label{gr1}
\big|\n \log w(\cdot,P_t)\big|^2\geq p\big((c+\de)^{-\f{2}{p}}-1\big):=C_2(p,\de).
\end{equation}

By (\ref{w3}),
$$\aligned
\n w(\cdot,P_t)&=\cos t\n x_1+\sin t\n x_2=\cos t\n(r\cos\th)+\sin t\n(r\sin\th)\\
               &=\cos(t-\th)\n r+r\sin(t-\th)\n \th
               \endaligned$$
and moreover by (\ref{fun}),
\begin{equation}\label{fun1}
\aligned
-\n\log w(\cdot,P_t)&=-w^{-1}(\cdot,P_t)\n w(\cdot,P_t)=-(c+\de)^{-1}\n w(\cdot,P_t)\\
                    &=-r^{-1}\n r+\tan (\th-t)\n \th
                    \endaligned
\end{equation}
on $F_t$. On the other hand, from (\ref{fun0}) we have
\begin{equation}\aligned\label{fun2}
\n F&=\n\th-\left(1-\left(\f{c+\de}{r}\right)^2\right)^{-\f{1}{2}}\f{c+\de}{r^2}\n r\\
    &=\n \th-r^{-1}\cot(\th-t)\n r.
 \endaligned
\end{equation}
Comparing (\ref{fun1}) with (\ref{fun2}) gives
\begin{equation}
-\n\log w(\cdot,P_t)=\tan(\th-t)\n F.
\end{equation}
$c+2\de\leq r\leq 1$ gives $\th-t=\arccos\big(\f{c+\de}{r}\big)\in \big[\arccos\big(\f{c+\de}{c+2\de}\big),\arccos(c+\de)\big]$; hence
$-\n\log w(\cdot,P_t)$ and $\n F$ point in the same direction, and
\begin{equation}\label{gr2}
\f{|\n\log w(\cdot,P_t)|}{|\n F|}=\tan(\th-t)\leq \f{\sqrt{1-(c+\de)^2}}{c+\de}:=C_3(\de).
\end{equation}

Now we put $H(\cdot,t)=-\log w(\cdot,P_t)$, then by (\ref{gr1}) and (\ref{gr2}), one can apply Lemma \ref{level0} to get the following conclusion

\begin{pro}\label{Hess1}
For a compact set $K\subset \Bbb{W}_c$, put $\de:=\f{1}{2}\inf_K(r-c)$ and $F:=\th-\arccos\big(\f{c+\de}{r}\big)$, then for arbitrary
$\ep>0$, there exists a smooth nonpositive function $\la$ on $K$, depending on $\ep$, such that
\begin{equation}\label{Hess2}
\Hess\ F\geq -C_3^{-1}\Hess\ \log w(\cdot,P_t)-\ep g+\la dF\otimes dF
\end{equation}
on the level set $F_t$. Here $g$ is the canonical metric on $\grs{n}{m}$ and $C_3$ is a positive constant depending only on $\de$.
\end{pro}

\begin{rem}
Put
$$\aligned
\phi(t)(a_ie_i+b_\a e_{n+\a})=&(a_1\cos t+b_1 \sin t)e_1+(-a_1\sin t+b_1\cos t)e_{n+1}\\
                               &+\sum_{i\geq 2}a_i e_i+\sum_{\a\geq 2}b_\a e_{n+\a},
                               \endaligned$$
then $t\in \R\mapsto \phi(t)$ is a 1-parameter subgroup of $SO(n+m)$, which induces a 1-parameter
isometry group of $\grs{n}{m}$. Since $w(\cdot,\cdot)$ is invariant under the action of $SO(n+m)$,
it is easily-seen that $F_t$ is the orbit of $F_0$ under the action of $\phi(t)$. This observation
can explain why the level sets of $F$ share similar properties in
terms of exterior geometry.
\end{rem}

Combining Propositions \ref{rank}-\ref{subhar3} and Proposition \ref{Hess1}, we can find
a strongly subharmonic function on a submanifold with parallel mean curvature under the assumptions on the Gauss image.

\begin{pro}\label{subhar1}

Let $M$ be an oriented submanifold in $\R^{n+m}$ with parallel mean curvature, if its Gauss image is contained
in a compact subset $K$ of $\Bbb{W}_{1/3}$, then there exists a positive bounded function $f$ on $M$,
such that
\begin{equation}\label{La1}
\De f\geq K_0|B|^2
\end{equation}
with a positive constant $K_0$ depending only on $K$. Moreover, if $M$ is minimal and the rank
of the Gauss map is not larger than $2$, then the above conclusion still holds when $\Bbb{W}_{1/3}$ is replaced by
$\Bbb{W}_0$.

\end{pro}

\begin{proof}

We only prove the general case, since the proof for the  case with the
additional assumption on the rank of the Gauss map
is similar.

Using the definition of $F$  in (\ref{fun0}) with $c=\f{1}{3}$ and combining (\ref{Hess2}) and (\ref{gr2}) gives
\begin{equation}\aligned
&\mu_0^{-1}\exp(\mu_0 F)\Hess \exp(\mu_0 F)\\
\geq &-C_3^{-1}\Hess \log w(\cdot,P_t)-\ep g+(\la+\mu_0)dF\otimes dF\\
\geq &-C_3^{-1}\Hess \log w(\cdot,P_t)-\ep g+C_3^{-2}|\la+\mu_0|d\log w(\cdot,P_t)\otimes d\log w(\cdot,P_t)
\endaligned
\end{equation}
where $\mu_0$ and $\ep$ are positive constants to be chosen. Define
\begin{equation}
f=\exp(\mu_0 F)\circ \g
\end{equation}
with $\g$ the Gauss map into $K\subset \Bbb{W}_{1/3}$, then using the composition formula implies
\begin{equation}\label{La5}\aligned
\mu_0^{-1}f^{-1}\De f=&\mu_0^{-1}\exp(\mu_0 F)\Big(\Hess\ \exp(\mu_0 F)(\g_* e_i,\g_* e_i)+d\exp(\mu_0 F)(\tau(\g)\Big)\\
                     \geq& -C_3^{-1}\Hess\ \log w(\cdot,P_t)(\g_* e_i,\g_* e_i)-\ep\lan \g_*e_i,\g_* e_i\ran\\
                         &+C_3^{-2}|\la+\mu_0|\sum_i \Big(\n_{e_i}\big|\log w(\cdot,P_t)\circ\g\big|\Big)^2\\
                     =&-C_3^{-1}\De\big(\log w(\cdot,P_t)\circ \g\big)-\ep|B|^2+C_3^{-2}|\la+\mu_0|\Big|\n\big(\log w(\cdot,P_t)\circ \g\big)\Big|^2
                     \endaligned
\end{equation}
on the level set $\{x\in M:f(x)=\exp(\mu_0 t)\}$ with the tension field $\tau(\g)$ of the Gauss map $\g$.
Here $\{e_i\}$ is an orthonormal basis of the tangent space at the considered
point of $M$, and we have used the harmonicity of the Gauss map. Since $ w(\cdot,P_t)$ equals $\f{1}{3}+\de$ everywhere on the level set
$F_t$ (see (\ref{fun})), one can apply Proposition \ref{subhar3} to obtain
\begin{equation}\aligned
\De \big(\log w(\cdot,P_t)\circ \g\big)-\Big|\n\big(\log w(\cdot,P_t)\circ \g\big)\Big|^2
\leq -C_1|B|^2
\endaligned
\end{equation}
with a positive constant $C_1$ depending only on $\de$.
Substituting this into (\ref{La5}) yields
$$\mu_0^{-1}f^{-1}\De f\geq (C_3^{-1}C_1-\ep)|B|^2+\big(C_3^{-2}|\la+\mu_0|-C_3^{-1}\big)\Big|\n\big(\log w(\cdot,P_t)\circ \g\big)\Big|^2.$$
By letting $\ep:=\f{1}{2}C_3^{-1}C_1$ and $\mu_0:=\sup_K |\la|+C_3$ we arrive at (\ref{La1}) with $K_0=\f{1}{2}C_3^{-1}C_1\mu_0\exp(\mu_0\inf_K F)$.

\end{proof}

We can immediately get the following Bernstein type theorem with the aid of the strongly subharmonic functions.

\begin{thm}
If $M$ is a parabolic, oriented minimal surface in $\R^{2+m}$, and if its
Gauss image is contained in a compact subset of $\Bbb{W}_0$, then
$M$ has to be an affine linear subspace.
\end{thm}

\begin{proof}

One can find a strongly subharmonic function $f$ on $M$ satisfying
(\ref{La1})  by applying Proposition \ref{subhar1}, since the mean curvature vector
field vanishes and the rank of the Gauss map is $\le 2$. Since a parabolic surface cannot admit any nonconstant bounded subharmonic function, $f$ has to be constant; hence
$\De f\equiv 0$ and moreover $|B|^2\equiv 0$, which means $M$ is totally geodesic.

\end{proof}

We now study the extrinsic rigidity problem, initiated by J. Simons \cite{Si}, finding better conditions on the Gauss image for
a compact minimal submanifold $M$ in $S^{n+m}$ to be an equator.

As pointed out and utilized by J. Simons \cite{Si}, the properties of the (minimal) submanifolds $M$ in the sphere are closed
related to those of the cone $CM$ generated by $M$, which is the image under the map of $M\times [0,+\infty)\ra \R^{n+m+1}$ defined
by $(x,t)\mapsto tx$. To avoid the singularity $t=0$, we also consider the truncated cone $CM_\ep$ ($\ep>0$) that is the image
of $M\times (\ep,+\infty)$ under the same map.

We choose a local orthonormal frame field $\{e_i,\nu_\a\}$ of $S^n$ along $M$, then by parallel translating along rays
issuing from the origin we obtain local vector fields $\{E_i\}$ and $\{\mathcal{N}_\a\}$; obviously $E_i=\f{1}{t}e_i$
and $\mathcal{N}_\a=\f{1}{t}\nu_\a$. Let $\tau$ be the unit tangent vector
along the rays, i.e. $\tau=\f{\p}{\p t}$, then $\{E_i,\mathcal{N}_\a,\tau\}$ is a local orthonormal frame field in
$\R^{n+m+1}$ and $\{E_i,\tau\}$ is a frame field in $CM_\ep$.

Let $B$ and $B^c$ denote the second fundamental form of $M\subset S^{n+m}$ and $CM_\ep\subset \R^{n+m+1}$, respectively. Then
a straightforward calculation shows (see \cite{x} p.18)
\begin{equation}\label{sec}
\lan B^c(E_i,E_j),\mathcal{N}_\a\ran=\f{1}{t}\lan B(e_i,e_j),\nu_\a\ran
\end{equation}
and
\begin{equation}
B^c(E_i,\tau)= B^c(\tau,\tau)=0.
\end{equation}
Hence $CM_\ep$ is a minimal submanifold in $\R^{n+m+1}$ if and only if $M$ is a minimal submanifold in $S^{n+m}$.

The \textit{normal Gauss map} $\g^N: M\ra \grs{m}{n+1}$ is defined by
$$\g^N(x)=N_x M$$
via  parallel translation in $\R^{n+m+1}$. Let $\eta$ be the natural
isometry between $\grs{n+1}{m}$ and  $\grs{m}{n+1}$ which maps each $(n+1)$-dimensional
oriented linear subspace to its orthogonal complementary $m$-space. It is easily seen that the Gauss map $\g: CM_\ep\ra \grs{n+1}{m}$
is a \textit{cone-like map}; more precisely, for every $t\in (\ep,+\infty)$ and $x\in M$,
$$(\eta\circ\g)(tx)=\g^N(x).$$
Thus $\g_*\tau=0$ and the rank of $\g$ at $tx$ equals the rank of $\g^N$ at $x$ for all $x\in M$ and any $t\in (\ep,+\infty)$.

One can assume $\n e_i=0$ at the considered point without loss of generality, then by computing (see \cite{x} p.18),
\begin{equation}
\n^c_{E_i} E_j=-\f{1}{t}\de_{ij}\tau
\end{equation}
with $\n^c$ the Levi-Civita connection on $CM_\ep$. Let $f$ be a cone-like real function on $CM_\ep$, then
\begin{equation}\label{La6}
\De^c f=\n^c_{E_i}\n^c_{E_i}f-(\n_{E_i}^c E_i)f=\f{1}{t^2}\n_{e_i}\n_{e_i}f_1+\f{1}{t}\n_{\tau}^c f=\f{1}{t^2}\De f_1.
\end{equation}
Here $\De^c$ is the Laplace-Beltrami operator on $CM_\ep$ and $f_1$ is a function on $M$ satisfying $f(tx)=f_1(x)$ for every
$x\in M$ and $t\in (\ep,+\infty)$.

Based on the strongly subharmonic functions constructed in Proposition \ref{subhar1}, we can derive the following extrinsic rigidity result.

\begin{thm}\label{Ber1}
Let $M^n$ be a compact, oriented minimal submanifold in $S^{n+m}$, $P,Q\in \grs{m}{n+1}$ that are S-orthogonal to each other.
Assume
\begin{equation}
w( N,P)^2+w(N,Q)^2> \left\{\begin{array}{cc}
0 & \text{if rank}(\g^N)\leq 2\\
\f{1}{9} & \text{otherwise}
\end{array}\right.
\end{equation}
holds for all normal $m$-vectors $N$ of $M$, and there exists no point $x\in M$, such that $w(N,Q)=0$ and $w(N,P)<0$ at $x$.
Then $M$ has to be an equator.
\end{thm}

\begin{proof}
As above, we only give the proof for the general case without the
assumptions on the rank of the Gauss map.

By the definition of $\eta$, we have
$$\psi\big(\eta(S)\big)=*\big(\psi(S)\big)$$
with $\psi$ denoting the Pl\"ucker embedding and $*$ being the Hodge
star operator. $*$ keeps the inner product invariant, i.e.
$$\lan *A,*B\ran=\lan A,B\ran\qquad \text{for }A,B\in \La^m(\R^{n+m+1}).$$
Hence
$$w\big(\eta(S_1),\eta(S_2)\big)=w(S_1,S_2)\qquad \text{for }S_1,S_2\in \grs{m}{n+1}.$$
If $S_1,S_2$ are S-orthogonal to each other, then $w(S_1,S_2)=0$ and $\dim(S_1\cap S_2)=m-1$, which implies
$w\big(\eta(S_1),\eta(S_2)\big)=0$ and $\dim\big(\eta(S_1)+\eta(S_2)\big)=n+m+1-\dim(S_1\cap S_2)=n+2$, hence
$\eta(S_1)$ and $\eta(S_2)$ are S-orthogonal to each other (see Proposition \ref{orthogonal}).

Denote
$$\Bbb{W}_{\f{1}{3}}=\Big\{S\in \grs{n+1}{m}:\big(w(S,\eta(P)),w(S,\eta(Q))\big)\in \Bbb{D}\backslash \big(\overline{\Bbb{D}}_{\f{1}{3}}\cup \{(a,0):a\leq 0\}\big)\Big\},$$
then the assumption on the normal Gauss map $\g^N$ implies that the
image under the Gauss map of $CM_\ep$ is contained in a compact subset
of $\Bbb{W}_{1/3}$. This enables us to find a strongly subharmonic function $f$ satisfying
\begin{equation}\label{La7}
\De^c f\geq K_0|B^c|^2.
\end{equation}
Noting that $f$ is the composition of the Gauss map $\g$ and a function on $K$, it should be a cone-like function. Denote $f_1=f|_M$, then combining
(\ref{La7}), (\ref{La6}) and (\ref{sec}) gives
\begin{equation}
\De f_1=t^2 \De^c f\geq K_0 t^2|B^c|^2=K_0|B|^2.
\end{equation}
Integrating both sides along $M$ implies
$$0=\int_M\De f_1*1\geq K_0\int_M |B|^2*1.$$
Thus $|B|^2\equiv 0$ and $M$ has to be totally geodesic.

\end{proof}

\begin{rem}
Obviously $w(N,P)>0$ implies $w(N,P)^2+w(N,Q)^2>0$. Hence for the case of dimension 2, our result is an improvement of extrinsic rigidity
theorems of Barbosa \cite{b} and Fischer-Colbrie \cite{fc}.
\end{rem}
\bigskip\bigskip

\Section{Curvature estimates}{Curvature estimates}

In the previous section, we have found a strongly subharmonic function
$f$ on a submanifold $M^n\subset \R^{n+m}$ with parallel mean curvature under the
assumption on the Gauss image. To make Stokes' theorem applicable, we take a $H^{1,2}$-function $\eta$ with compact support; multiplying
$\eta$ with both sides of (\ref{La1}) and then integrating it along $M$ gives
\begin{equation}\label{int1}\aligned
K_0\int_M |B|^2\eta*1\leq& \int_M \eta\De f*1=\int_M \big(\text{div}(\eta\n f)-\lan \n\eta,\n f\ran\big)*1\\
                      =&-\int_M \lan \n\eta,\n f\ran*1.\endaligned
\end{equation}
In order to obtain a-priori estimates for $|B|^2$, we have to choose a 'good' test function $\eta$.

We shall use the Green test function technique employed in \cite{h-k-w} and \cite{h-j-w}. Unfortunately, this method cannot be applied for arbitrary
noncompact Riemannian manifolds. We have to impose a so-called DVP-condition as in \cite{j-x-y}.

\begin{defi}
Let $(M,g)$ be a Riemannian manifold (not necessarily complete)
satisfying: \newline
(D) There is a distance function $d$ on $M$, such that the metric
topology induced by $d$ is equivalent to the Riemannian topology of $M$, and $d(x,y)\leq r(x,y)$
for any $x,y\in M$, where $r$ is the distance function induced by the metric tensor of $M$;\newline
(V) Let $B_R(x_0)=\{x\in M: d(x,x_0)<R\}$, $V(x_0,R)=\text{Vol}\big(B_R(x_0)\big)$, then there exists a positive constant $K_1$ not depending
on $R$ and $x_0$, such that
$$V(x_0,2R)\leq K_1\ V(x_0,R)\qquad \text{whenever }B_{2R}(x_0)\subset\subset M;$$
(P) For any $x_0\in M$ and $R>0$ satisfying $B_R(x_0)\subset\subset M$, the following Neumann-Poincar\'{e} inequality
$$\int_{B_R(x_0)}|v-\bar{v}_{R}|^2*1\leq K_2 R^2\int_{B_R(x_0)}|\n v|^2*1$$
holds with $\bar{v}_R$ denoting the average value of $v$ in
$B_R(x_0)$ and $K_2$ being a positive constant not depending on
$x_0$ and $R$.\newline Then we say that $(M,g)$ satisfies a
DVP-condition.
\end{defi}

\begin{rem}\label{DVP}
If $M=\Bbb{D}^n(r_0)\subset \R^n$ with metric $g=g_{ij}dx^idx^j$, and the eigenvalues of $(g_{ij})$ are uniformly bounded from below by
a positive constant $\la$ and from above by a positive constant $\mu$, then by denoting $d(x,y)=\la|x-y|$ with $|\cdot|$ the standard
Euclidean norm, we can show $(M,g)$ satisfies the DVP-condition with $K_1=\big(\f{4\mu}{\la}\big)^n$ and $K_2=4\pi^{-2}\big(\f{\mu}{\la}\big)^{n+2}$
(see \cite{j-x-y}). Thus every Riemannian manifold has a coordiate
chart centered at any point that satisfies the DVP-condition. If $r_0=+\infty$,
we call $M$  a simple manifold. Other complete Riemannian manifolds
satisfying a DVP-condition include $(M,g)$ with nonnegative Ricci
curvature (see \cite{bu} for the proof of the Neumann-Poincar\'{e} inequality) and area-minimizing embedded hypersurfaces in $\R^{n+1}$. (Here $d$ is
taken to be the extrinsic distance function, see \cite{b-g} for the
proof of the Neumann-Poincar\'{e} inequality.)
\end{rem}

Based on the DVP-condition, one can apply Moser's iteration \cite{m}
to derive the Harnack inequality for superharmonic functions
on $M$ (see \cite{j-x-y} Lemma 4.1), which implies the following estimates:

\begin{lem}(\cite{j-x-y}, Corollary 4.1)\label{Har}
Let $M$ be a Riemannian manifold satisfying a DVP-condition, then there exists a constant $C_4\in (0,1)$, only depending on $K_1$
and $K_2$, such that
\begin{equation}
v_{+,\f{R}{2}}\leq (1-C_4)v_{+,R}+C_4 \bar{v}_{\f{R}{2}}
\end{equation}
for any subharmonic function $v$ on $B_R(x_0)$ ($B_{2R}(x_0)\subset\subset M$). Here $v_{+,R}$ denotes the supremum of
$v$ on $B_R(x_0)$.
\end{lem}

Let $\Om$ be an open domain of $M$ and $\rho>0$. Using the Lax-Milgram
Theorem, it is easy to prove that there is a unique
$H_0^{1,2}(\Om)$-function $G^\rho(\cdot,x)$, such that
\begin{equation}
\int_\Om \lan \n G^\rho(\cdot,x), \n\phi\ran*1=\aint{B_\rho(x)}\phi\qquad \text{for any }\phi\in H_0^{1,2}(\Om)
\end{equation}
whenever $B_\rho(x)\subset \Om$. Here
$$\aint{B_\rho(x)}\phi=\f{\int_{B_\rho(x)}\phi}{V(x,\rho)}.$$
$G^\rho$ is called the \textit{mollified Green function}. As in
\cite{g-w} and \cite{j-x-y}, we can obtain  a-priori estimates for mollified
Green functions:

\begin{lem} (\cite{j-x-y}, Lemma 4.3)\label{Green}
Let $(M,g)$ be a Riemannian manifold (dimension $\geq 3$) satisfying a
DVP-condition. Let $x_0\in M$ and $R>0$ satisfying $B_{2R}(x_0)\subset\subset M$.
Denote
\begin{equation}
\om^R=\f{V(x_0,\f{R}{2})}{R^2}G^{\f{R}{2}}(\cdot,x_0)
\end{equation}
with $G^{\f{R}{2}}$ being the mollified Green function on $B_R(x_0)$, then
\begin{equation}
\om^R\leq c_1\qquad \text{on }B_R(x_0)
\end{equation}
and
\begin{equation}
\om^R\geq c_2\qquad \text{on }B_{\f{R}{2}}(x_0)
\end{equation}
with $c_1$ and $c_2$ denoting positive constants depending only on $K_1$ and $K_2$, not depending on $x_0$ and $R$.
\end{lem}

Choosing $\eta=(\om^R)^2$ as a test function in (\ref{int1}), we can derive the following curvature estimates with the aid of
Lemmas \ref{Har} and \ref{Green}. See \cite{j-x-y} and \cite{j-x-y2} for  details.

\begin{lem}(Curvature estimates)\label{cur}
Let $M^n$ be a submanifold in $\R^{n+m}$ ($n\geq 3$) with parallel mean curvature. Assume there is a distance function
$d$ on $M$, the metric ball $B_{4R_0}(x_0)$ defined by $d$ satisfies
the DVP-condition for some $x_0\in M$ and $R_0\in (0,+\infty]$,
and the Gauss image of $B_{4R_0}(x_0)$ is contained in a compact set $K\subset \Bbb{W}_{1/3}$. Then there exists a positive constant
$C_5$, depending only on $K,K_1,K_2$, such that for arbitrary $R<R_0$,
\begin{equation}
\f{R^2}{V(x_0,R)}\int_{B_R(x_0)}|B|^2*1\leq C_5(f_{+,R}-f_{+,\f{R}{2}})
\end{equation}
with $f$ being the strongly subharmonic function constructed in Proposition \ref{subhar1}. Moreover, there exists a positive constant
$C_6$ depending only on $K,K_1,K_2$, such that for arbitrary $\ep>0$, we can find $R\in \big[\exp(-C_6\ep^{-1})R_0,R_0\big]$, such that
\begin{equation}
\f{R^2}{V(x_0,R)}\int_{B_R(x_0)}|B|^2*1\leq \ep.
\end{equation}
If $M$ is a minimal submanifold with $rank(\g)\leq 2$,
then the condition on the Gauss image
can be relaxed to $\g(B_{4R_0}(x_0))\subset K\subset \Bbb{W}_0$.
\end{lem}
\bigskip\bigskip

\Section{Gauss image shrinking lemmas  and Bernstein type theorems}{Gauss image shrinking lemmas  and Bernstein type theorems}

Let $u$ be a harmonic map from $M$ and $H$ be a smooth function on the target manifold, then $H\circ u$ defines
a smooth function on $M$. Take a cut-off function $\eta$ supported in the interior of $B_R(x_0)$, $0\leq\eta\leq 1$, $\eta\equiv 1$
on $B_{\f{R}{2}}(x_0)$ and $|\n\eta|\leq c_0R^{-1}$; and denote by $G^\rho$ the mollified Green function in
$B_R(x_0)$. Then by Stokes' theorem
\begin{equation}
\int_{B_R(x_0)}\text{div}\big(\eta G^\rho(\cdot,x)\n (H\circ u)\big)*1=0\qquad\text{whenever }B_\rho(x)\subset B_R(x_0).
\end{equation}
After careful calculation as in \cite{j-x-y}\cite{j-x-y2}, we can arrive at a pointwise estimate for $H\circ u$:
\begin{pro}
Let $M$ be as in Lemma \ref{cur}. Let $u$ be a harmonic map of $B_{4R_0}(x_0)\subset M$ into a convex domain $V\subset \R^l$
with metric $h_{ij}dx^i dx^j$, such that
$$K_3|\xi|^2\leq \xi^i h_{ij}\xi^j\leq K_4|\xi|^2$$
everywhere on $V$. Then for any smooth function $H$ on $V$,
\begin{equation}\aligned\label{es1}
H\circ u(x)\leq& H(\bar{u}_R)+C_7\sup_V|\n H|\Big(\f{R^2}{V(x_0,R)}\int_{B_R(x_0)}|du|^2*1\Big)^{\f{1}{2}}\\
               &-\liminf_{\rho\ra 0^+}\int_{B_R(x_0)}G^\rho(\cdot,x)\eta\De (H\circ u)*1
               \endaligned
\end{equation}
for all $x\in B_{\f{R}{4}}(x_0)$ with $R<R_0$. Here $\bar{u}_R\in \overline{V}$ (the closure of $V$) denotes the average value of $u$ in $B_R(x_0)$ and $C_7$ is a positive constant
depending only on $K,K_1,K_2,K_3,K_4$, but independent of $H$, $x_0$ and $R$.

\end{pro}

If we take $u$ to be the harmonic Gauss map, then
$|du|^2=|d\g|^2=|B|^2$, and Lemma \ref{cur} implies that the second
term on the
right hand side of (\ref{es1}) becomes arbitrarily small when $R$ is
sufficiently small. Hence, for proving the Gauss image shrinking property as in \cite{j-x-y}\cite{j-x-y2} it remains to show:
\begin{itemize}
\item $\Bbb{W}_0$ is diffeomorphic to a convex domain in Euclidean space;
\item For any given $S\in K\subset \Bbb{W}_{1/3}$, one can find a smooth function $H$, such that $H\circ \g$
is a subharmonic function on $B_R(x_0)$, and the sublevel set $\big(H\leq H(S)\big)$ is contained in a matrix coordinate chart
of some point in $\grs{n}{m}$.
\end{itemize}

In the sequel,  $x_1,x_2,r,\th, e_i,e_{n+\a}$ and $P_t$  will be as in Section \ref{s3}.

\begin{lem}\label{target}
There is a diffeomorphism $\phi=(\phi_1,\phi_2): \Bbb{W}_0\ra (-\pi,\pi)\times \R^{nm-1}$ defined by
$$S\mapsto \big(\phi_1(S),\phi_2(S)\big),$$
such that
\begin{equation}\aligned
&\phi_1=\th,\\
&\big|\phi_2(S)\big|=\f {1}{r}-1.
\endaligned
\end{equation}

\end{lem}

\begin{proof}
As shown in Section \ref{s3}, $\th$ is a smooth $(-\pi,\pi)$-valued
function on $\Bbb{W}_0$, and we shall investigate the level sets
of $\th$.

For any $t\in (-\pi,\pi)$, when $\th(S)=t$ we have $w(S,P_{t+\f{\pi}{2}})=0$ and $w(S,P_t)>0$ by (\ref{w3}). Now we write
$$y_1=w(\cdot,P_t),\qquad y_2=w(\cdot,P_{t+\f{\pi}{2}}),$$
then
\begin{equation}
y_1=\cos t\ x_1+\sin t\ x_2\qquad y_2=-\sin t\ x_1+\cos t\ x_2.
\end{equation}
Denote $\Bbb{U}_t=\big\{S\in \grs{n}{m}:w(S,P_t)>0\big\}$. Let $\chi_t:\Bbb{U}_t\ra M_{n\times m}$ be the diffeomorphism
mapping $S\in \Bbb{U}_t$ to its $(n\times m)$-matrix coordinate, then
the Pl\"ucker coordinate of $S$ is
\begin{equation}
\psi(S)=(f_1+Z_{1\a}f_{n+\a})\w\cdots\w (f_n+Z_{n\a}f_{n+\a})\qquad \text{with } Z=\chi_t(S).
\end{equation}
Here $f_1=\cos t\ e_1+\sin t\ e_{n+1}$, $f_{n+1}=-\sin t\ e_1+\cos t\ e_{n+1}$, $f_i=e_i$ for all $2\leq i\leq n$ and $f_{n+\a}=e_{n+\a}$
for every $2\leq \a\leq m$. Noting that $\psi(P_t)=f_1\w\cdots\w f_n$ and $\psi(P_{t+\f{\pi}{2}})=f_{n+1}\w f_2\w\cdots\w f_n$.
A direct calculation shows that for $S\in U_t$
\begin{equation}
y_1(S)=w(S,P_t)=\f{\lan \psi(S),\psi(P_t)\ran}{\lan \psi(S),\psi(S)\ran^{\f{1}{2}}\lan \psi(P_t),\psi(P_t)\ran^{\f{1}{2}}}=(I_n+ZZ^T)^{-\f{1}{2}},
\end{equation}
and
\begin{equation}\label{y2}
y_2(S)=w(S,P_{t+\f{\pi}{2}})=\f{\lan \psi(S),\psi(P_{t+\f{\pi}{2}})\ran}{\lan \psi(S),\psi(S)\ran^{\f{1}{2}}\lan \psi(P_{t+\f{\pi}{2}}),\psi(P_{t+\f{\pi}{2}})\ran^{\f{1}{2}}}=Z_{11}(I_n+ZZ^T)^{-\f{1}{2}}.
\end{equation}
Hence, $y_2(R)=0$ if and only if $Z_{11}=0$, i.e.
\begin{equation}\label{th1}
\{S\in \Bbb{W}_0:\th(S)=t\}=\chi_t^{-1}\{Z\in M_{n\times m}:Z_{11}=0\}.
\end{equation}

A straightforward calculation gives
$$\aligned
\n y_1&=\cos(t-\th)\n r+r\sin(t-\th)\n \th,\\
\n y_2&=-\sin(t-\th)\n r+r\cos(t-\th)\n \th.
\endaligned$$
Hence, along the level set $(\th=t)$
$$r=y_1,\quad \n y_1=\n r, \quad\n y_2=r\n \th.$$  Differentiating both sides
of (\ref{y2}) and noting (\ref{th1}), we have
\begin{equation}
\n \th=r^{-1}\n y_2=r^{-1}\n(Z_{11}y_1)=r^{-1}(y_1 \n Z_{11}+Z_{11}\n y_1)=\n Z_{11},
\end{equation}
whenever $\th(S)=t.$  In $U_t$ let $E_{i\a}$ be the matrix with 1 in the intersection of row
$i$ and column $\a$ and 0 otherwise. Denote $g_{i\a,j\be}=\lan E_{i\a},E_{j\be}\ran$ and let $(g^{i\a,j\be})$ be the inverse matrix of $(g_{i\a,j\be})$. A standard computation
shows $|\n Z_{11}|^2=g^{11,11}$. It has been shown in \cite{x-y1} that the eigenvalues of $(g^{i\a,j\be})$ are
$$\big\{(1+\la_i^2)(1+\la_\a^2):1\leq i\leq n,1\leq \a\leq m\big\}$$
with $\la_i=\tan \th_i$ and $\{\th_i\}$ being the Jordan angles between $S$ and $P_t$. Hence
$$1\leq g^{11,11}\leq \max_{(i,\a)}\big((1+\la_i^2)(1+\la_\a^2)\big)\leq w^{-4}(\cdot,P_t)=r^{-4};$$
i.e.
\begin{equation}\label{th2}
1\leq |\n \th|^2\leq r^{-4}.
\end{equation}

From (\ref{th1}) and (\ref{th2}), we see that $\th$ is a non-degenerate function on $\Bbb{W}_0$, and each level set
of $\th$ is diffeomorphic to $\R^{nm-1}$. $T:M_{n\times m}\ra
M_{n\times m}$ defined by
$$Z\mapsto \Big(\det(I_n+ZZ^T)^{\f{1}{2}}-1\Big)\f{Z}{\big(\text{tr}(ZZ^T)\big)^{\f{1}{2}}}$$
obviously  is a diffeomorphism. Note that here $\big(\text{tr}(ZZ^T)\big)^{\f{1}{2}}=\big(\sum_{i,\a}Z_{i\a}^2\big)^{\f{1}{2}}$
equals the Euclidean norm of $Z$ when $Z$ is treated as a vector in $\R^{nm}$. Thus, $T\circ \chi_t$ is a diffeomorphism between the level set
$(\th=t)$ and  $\R^{nm-1}$, moreover
$$|T\circ \chi_t|=w^{-1}(\cdot,P_t)-1=r^{-1}-1.$$
Therefore,  $\phi_1=\th$ and $\phi_2=T\circ \chi_t$
are the required mappings.

\end{proof}
\bigskip

\begin{pro}\label{subhar4}
For any $c\in (\f{1}{3},1)$ and any compact subset $\Th$ of $(-\pi,\pi)$, there exists a smooth family of nonnegative, smooth functions
$H(\cdot,t)(t\in \Th)$ on
$$\overline{\Bbb{W}}_{c,\Th}:=\{S\in \Bbb{W}_0:r(S)\geq c,\th(S)\in \Th\},$$
such that\newline
(i) $H(S,t)=0$ if and only if $S=P_t$;\newline
(ii) $H(S,t)\leq 1$ (or $H(S,t)<1$) if and only if $w(S,P_t)\geq \f{3}{4}c+\f{1}{12}$ (or $w(S,P_t)>\f{3}{4}c+\f{1}{12}$);\newline
(iii) For any submanifold $M$ in $\R^{n+m}$ with parallel mean curvature, if the Gauss image of $M$ is contained in $\overline{\Bbb{W}}_{c,\Th}$,
then $H(\cdot,t)\circ \g$ is a subharmonic function on $M$ for all $t\in \Th$.

\end{pro}

\begin{pro}\label{subhar5}
For any $c\in (0,1)$ and any compact subset $\Th$ of $(-\pi,\pi)$, there exists a smooth family of nonnegative, smooth functions
$H(\cdot,t)(t\in \Th)$ on
$\overline{\Bbb{W}}_{c,\Th},$
such that\newline
(i) $H(S,t)=0$ if and only if $S=P_t$;\newline
(ii) $H(S,t)\leq 1$ (or $H(S,t)<1$) if and only if $w(S,P_t)\geq \f{3}{4}c$ (or $w(S,P_t)>\f{3}{4}c$);\newline
(iii) If $M$
is a minimal submanifold in $\R^{n+m}$  with $rank(\g)\leq 2$,
then $H(\cdot,t)\circ \g$ is a subharmonic function on $M$ for all $t\in \Th$ if the Gauss image of $M$ is contained
in $\overline{\Bbb{W}}_{c,\Th}$.
\end{pro}

\begin{proof}
We only prove Proposition \ref{subhar4}, since the proof of Proposition \ref{subhar5} is similar.

Let $\varphi$ be a smooth function on $[0,+\infty)$ satisfying
$$\left\{\begin{array}{ll}
\varphi(u)=0,& u\in [0,\f{11}{12}-\f{3}{4}c];\\
\varphi(u)=u-\f{5}{6}+\f{1}{2}c,& u\in [\f{3}{4}-\f{1}{4}c,+\infty);\\
0\leq \varphi'\leq 1. &
\end{array}\right.$$
Then one can define $\Psi$ on $\overline{\Bbb{W}}_{c,\Th}\times (-\pi,\pi)\times [0,+\infty)$ by
\begin{equation}
(S,t,u)\mapsto \left\{\begin{array}{cc}
r\cos(\th-t-\varphi(u))+u-\varphi(u)-1 & \text{if }\th(S)\geq t;\\
r\cos(\th-t+\varphi(u))+u-\varphi(u)-1 & \text{if }\th(S)<t.
\end{array}\right.
\end{equation}

Now we fix $t$ and denote $\Psi_t(S,u)=\Psi(S,t,u)$. For arbitrary $S\in \overline{\Bbb{W}}_{c,\Th}$, we put
$$I_S:=\Big\{u\in [0,+\infty):\max\{0,|\th-t|-\pi\}\leq \varphi(u)\leq |\th-t|\Big\},$$
then obviously $I_S$ is a closed interval, $I_S:=[m_S,M_S]$. If $\th(S)\geq t$, then
$$\p_2\Psi_t=1-\big(1+r\sin(\th-t-\varphi(u))\big)\varphi'(u)\geq 0$$
and the equality holds if and only if $\varphi(u)=|\th-t|$ or $|\th-t|-\pi$ and $\varphi'(u)=1$; which implies
$\varphi(u)=m_S$ or $M_S$. Thus
\begin{equation}\label{d1}
\p_2 \Psi_t(S,\cdot)>0 \qquad \text{on }(m_S,M_S).
\end{equation}
Similarly (\ref{d1}) holds when $\th(S)<t$.

When $|\th(S)-t|\leq \pi$, we have $m_S=0$ and hence
\begin{equation}
\Psi_t(S,m_S)=r\cos(\th-t)-1\leq 0.
\end{equation}
Otherwise $|\th(S)-t|>\pi$ and $m_S$ satisfies $\varphi(m_S)=|\th-t|-\pi$, thus
\begin{equation}\aligned
\Psi_t(S,m_S)&=-r+m_S-\varphi(m_S)-1\leq -r-1+\lim_{u\ra +\infty}\big(u-\varphi(u)\big)\\
             &=-r-1+\f{5}{6}-\f{1}{2}c\leq -\f{3}{2}c-\f{1}{6}<0.
             \endaligned
\end{equation}
Here the first inequality follows from $\big(u-\varphi(u)\big)'=1-\varphi'(u)\geq 0$.

Obviously $\varphi(M_S)=|\th-t|$. By the definition of $\varphi$, $\varphi$ cannot be identically zero on a neighborhood
of $M_S$. Thus $M_S\geq \f{11}{12}-\f{3}{4}c$, and
\begin{equation}\label{Max}\aligned
\Psi_t(S,M_S)&=r+M_S-\varphi(M_S)-1\geq r-1+\f{11}{12}-\f{3}{4}c\\
             &\geq \f{1}{4}c-\f{1}{12}>0.
             \endaligned
\end{equation}

By (\ref{d1})-(\ref{Max}), for each $S\in \overline{\Bbb{W}}_{c,\Th}$, there exists a unique $\td{H}=\td{H}(S,t)\in [m_S,M_S)$, such that
\begin{equation}\label{tdH1}
\Psi\big(S,t,\td{H}(S,t)\big)=\Psi_t\big(S,\td{H}(S,t)\big)=0.
\end{equation}

Denote
\begin{equation}
\Om:=\big\{(S,t)\in \overline{\Bbb{W}}_{c,\Th}\times (-\pi,\pi):\th(S)\neq t\big\},
\end{equation}
then $\Psi$ is obviously smooth on $\Om\times [0,+\infty)$. Hence the  implicit theorem implies $\td{H}$
is smooth on $\Om$. To show smoothness of $\td{H}$, it remains to prove that $\td{H}$ is smooth on a neighborhood
of $\Om^c=\big\{(S,t)\in \overline{\Bbb{W}}_{c,\Th}\times (-\pi,\pi):\th(S)=t\big\}$. Denote
\begin{equation}
\Om_0:=\big\{(S,t)\in \overline{\Bbb{W}}_{c,\Th}:w(S,P_t)\geq \f{1}{12}+\f{3}{4}c\big\},
\end{equation}
then $\Om^c\subset\subset \Om_0$ since $\th(S)=t$ implies $w(S,P_t)=r\geq c>\f{1}{12}+\f{3}{4}c$.
For all $(S,t)\in \Om_0$, $1-w(S,P_t)\leq 1-(\f{1}{12}+\f{3}{4}c)=\f{11}{12}-\f{3}{4}c$ and hence
$\varphi\big(1-w(S,P_t)\big)=0$. Noting that $w(S,P_t)>0$ implies $|\th-t|<\pi$, we have
$1-w(S,P_t)\in [m_S,M_S)$, and moreover
$$\Psi_t\big(S,1-w(S,P_t)\big)=r\cos(\th-t)+\big(1-w(S,P_t)\big)-1=0.$$
Therefore
\begin{equation}\label{tdH}
\td{H}(S,t)=1-w(S,P_t)\qquad\text{for all }(S,t)\in \Om_0
\end{equation}
and the smoothness of $\td{H}$  follows.

Now we put
\begin{equation}
U_t:=\Big\{S\in \overline{\Bbb{W}}_{c,\Th}: w(S,P_t)\geq \f{1}{12}+\f{3}{4}c\Big\},
\end{equation}
then obviously $0\leq \td{H}(S,t)\leq \f{11}{12}-\f{3}{4}c$ whenever $S\in U_t$. On the other hand,
for arbitrary $S\in \overline{\Bbb{W}}_{c,\Th}\backslash U_t$, one of the following two cases has to occur: (I) $|\th(S)-t|\geq \pi$;
(II) $|\th(S)-t|\leq \pi$ and $w(S,P_t)<\f{1}{12}+\f{3}{4}c$. If Case (I) holds, then $\td{H}(S,t)\geq m_S>\f{11}{12}-\f{3}{4}c$.
For the second case,
since
$$\Psi_t\big(s,\f{11}{12}-\f{3}{4}c\big)=r\cos(\th-t)+\f{11}{12}-\f{3}{4}c-1<0,$$
one can deduce that $\td{H}(S,t)>\f{11}{12}-\f{3}{4}c$ due to the
monotonicity of $\Psi_t$ with respect to the $u$
variable. Therefore
\begin{equation}\label{tdH3}
U_t=\Big\{S\in \overline{\Bbb{W}}_{c,\Th}:\td{H}(S,t)\leq \f{11}{12}-\f{3}{4}c\Big\}.
\end{equation}
Similarly
\begin{equation}
\text{int}(U_t)=\Big\{S\in \overline{\Bbb{W}}_{c,\Th}:\td{H}(S,t)<\f{11}{12}-\f{3}{4}c\Big\}.
\end{equation}

It is easily seen from (\ref{tdH}) and (\ref{tdH3}) that
\begin{equation}\label{HessH1}
\Hess \td{H}(\cdot,t)\geq \big(\f{1}{12}+\f{3}{4}c\big)\big(-\Hess \log w(\cdot,P_t)+d\log w(\cdot,P_t)\otimes d\log w(\cdot,P_t)\big)\qquad\text{on }U_t.
\end{equation}

For each $a\geq \f{11}{12}-\f{3}{4}c$ and $S\in \overline{\Bbb{W}}_{c,\Th}$ satisfying $\th(S)>t$, $\td{H}(S,t)=a$ if and only if
$$0=\Psi_t(S,a)=r\cos(\th-t-\varphi(a))+a-\varphi(a)-1;$$
i.e.
\begin{equation}
w(S,P_{t+\varphi(a)})=1+\varphi(a)-a
\end{equation}
with
$$\f{1}{12}+\f{3}{4}c\geq 1+\varphi(a)-a\geq 1+\lim_{u\ra +\infty}(\varphi(u)-u)=\f{1}{6}+\f{1}{2}c>\f{1}{3}.$$
Hence
\begin{equation}
N_{t,a}^+:=\{S\in \overline{\Bbb{W}}_{c,\Th}:\th(S)>t,\td{H}(S,t)=a\}
\end{equation}
overlaps the level set $\big\{S\in \grs{n}{m}:-\log w(S,P_{t+\varphi(a)})=-\log(1+\varphi(a)-a)\big\}$.

From (\ref{tdH1}) we have $\td{\Psi}_t:=\Psi_t\big(S,\td{H}(S,t)\big)\equiv 0$. Differentiating both sides implies
\begin{equation}\label{tdH2}
0=\n_\nu \td{\Psi}_t=\n_\nu \Psi_t(\cdot,a)+(\p_2 \Psi_t)\n_\nu \td{H}(\cdot,t)\qquad \text{on }N_{t,a}^+
\end{equation}
with $\nu$ the unit normal vector field on $N_{t,a}^+$. Noting that $\Psi_t(\cdot,a)=w(\cdot,P_{t+\varphi(a)})-(1+\varphi(a)-a)$
and $\p_2\Psi_t>0$, (\ref{tdH2}) tells us $\n\td{H}(\cdot,t)$ and $-\n\log w(\cdot,P_{t+\varphi(a)})$ are both nonzero
normal vector fields on $N_{t,a}^+$, pointing in the same direction. The compactness of $\bigcup_{a\geq \f{11}{12}-\f{3}{4}c}N_{t,a}^+$
and
$\Th$ implies the existence of a positive constant $C_8$ not depending on $a$ and $t$, such that
\begin{equation}
\f{|\n \td{H}(\cdot,t)|}{\big|\n \log w(\cdot,P_{t+\varphi(a)})\big|}\geq C_8\qquad \text{on }N_{t,a}^+.
\end{equation}
Hence applying Lemma \ref{level0} gives
\begin{equation}\label{HessH2}
\Hess\ \td{H}(\cdot,t)\geq -C_8\log w(\cdot,P_{t+\varphi(a)})-\ep g+\la d\td{H}(\cdot,t)\otimes d\td{H}(\cdot,t)\qquad \text{on }N_{t,a}^+
\end{equation}
with $g$ being the canonical metric on $\grs{n}{m}$, $\ep$ being a positive constant to be chosen and $\la$ denoting
a continuous nonpositive function depending on $\ep$.

Similarly,
\begin{equation}
N_{t,a}^-:=\{S\in \overline{\Bbb{W}}_{c,\Th}:\th(S)<t,\td{H}(S,t)=a\}
\end{equation}
overlaps the level set $\big\{S\in \grs{n}{m}:\log w(S,P_{t-\varphi(a)})=-\log(1+\varphi(a)-a)\big\}$ for each
$a\geq \f{11}{12}-\f{3}{4}c$. On it $\n \td{H}(\cdot,t)$ and $-\n\log w(\cdot,P_{t-\varphi(a)})$ are both nonzero
normal vector fields pointing in the same direction. Again using Lemma \ref{level} yields
\begin{equation}\label{HessH3}
\Hess \td{H}(\cdot,t)\geq -C_8\log w(\cdot,P_{t-\varphi(a)})-\ep g+\la d\td{H}(\cdot,t)\otimes d\td{H}(\cdot,t)\qquad \text{on }N_{t,a}^-.
\end{equation}

Let $M$ be a submanifold in $\R^{n+m}$ with parallel mean curvature whose Gauss image is contained in $\overline{\Bbb{W}}_{c,\Th}$, denote
$$\td{h}(\cdot,t)=\exp(\mu_0\td{H}(\cdot,t))\circ \g$$
where $\mu_0$ is a positive constant to be chosen. Once $\g(x)\in U_t$, combining (\ref{HessH1}) and (\ref{La4}) gives
\begin{equation}\label{La8}\aligned
&\mu_0^{-1}\td{h}(\cdot,t)^{-1}\td{h}(\cdot,t)\\
\geq&-\big(\f{1}{12}+\f{3}{4}c\big)\De\Big(\log w(\cdot,P_t)\circ\g\Big)+\big(\f{1}{12}+\f{3}{4}c+\mu_0\big)\Big|\n\big(\log w(\cdot,P_t)\circ \g\big)\Big|^2\\
\geq&\big(\f{1}{12}+\f{3}{4}c\big)C_1|B|^2+\mu_0\Big|\n\big(\log w(\cdot,P_t)\circ \g\big)\Big|^2
\endaligned
\end{equation}
at $x\in M$, where $C_1$ is a positive constant depending only on $c$.
If $\g(x)\in N_{t,a}^+$, based on (\ref{HessH2}) and (\ref{La4}), one can proceed as in the proof of Proposition \ref{subhar1} to obtain
\begin{equation}
\mu_0^{-1}\td{h}(\cdot,t)^{-1}\De\td{h}(\cdot,t)\geq (C_8 C_1-\ep)|B|^2+(C_8^2|\la+\mu_0|-C_8)\big|\n (\log w(\cdot,P_{t+\varphi(a)})\circ\g)\big|^2
\end{equation}
at $x$. Similarly, once $\g(x)\in N_{t,a}^-$,
\begin{equation}\label{La9}
\mu_0^{-1}\td{h}(\cdot,t)^{-1}\De\td{h}(\cdot,t)\geq (C_8 C_1-\ep)|B|^2+(C_8^2|\la+\mu_0|-C_8)\big|\n (\log w(\cdot,P_{t-\varphi(a)})\circ\g)\big|^2.
\end{equation}
Now we put
\begin{equation}
\ep:=C_8C_1,\qquad \mu_0:=\sup|\la|+C_8^{-1}
\end{equation}
(where $\sup|\la|<+\infty$,  since $\la$ is a continuous function on a compact set $(\overline{\Bbb{W}}_{c,\Th}\times \Th)\backslash \text{int}(\Om_0)$),
then from (\ref{La8})-(\ref{La9}) we see $\td{h}$ is a subharmonic function on $M$. Therefore
\begin{equation}
H(\cdot,t):=\f{\exp\big(\mu_0\td{H}(\cdot,t)\big)-1}{\exp\big((\f{11}{12}-\f{3}{4}c)\mu_0\big)-1}
\end{equation}
are required functions satisfying (i)-(iii).
\end{proof}

\begin{rem}
The auxiliary function $\varphi$ in the above proof can be easily obtained from the standard bump function. Choose
$\xi_1$ to be a nonnegative smooth function on $\R$, whose support is $(\f{11}{12}-\f{3}{4}c,\f{3}{4}-\f{1}{4}c)$,
then
$$\xi_2(u):=\f{\int_{\f{11}{12}-\f{3}{4}c}^u \xi_1}{\int_{\f{11}{12}-\f{3}{4}c}^{\f{3}{4}-\f{1}{4}c}\xi_1}$$
is a smooth function on $\R$ satisfying $0\leq \xi_2\leq 1$, $\xi_2(u)=0$ whenever $u\leq \f{11}{12}-\f{3}{4}c$ and $\xi_2(u)=1$
whenever $u\geq \f{3}{4}-\f{1}{4}c$.  $\a\in (0,+\infty)\mapsto \int_{\f{11}{12}-\f{3}{4}c}^{\f{3}{4}-\f{1}{4}c} \xi_2^\a$ is
a strictly decreasing function, which converges to $0$ as $u\ra +\infty$ and converges to $-\f{1}{6}+\f{1}{2}c$ as $u\ra 0$. There exists
 a unique $\be\in (0,+\infty)$, such that
$$\int_{\f{11}{12}-\f{3}{4}c}^{\f{3}{4}-\f{1}{4}c} \xi_2^\be=-\f{1}{12}+\f{1}{4}c.$$
Then
$$\varphi(u):=\int_0^u \xi_2^\be$$
is the required auxiliary function.

\end{rem}

Based on Lemma \ref{target} and Proposition
\ref{subhar4}-\ref{subhar5}, one can derive a Gauss image shrinking property for submanifolds
with parallel mean curvature as in \cite{j-x-y}\cite{j-x-y2}:

\begin{lem}\label{sh1}
Let $M^n$ be a submanifold in $\R^{n+m}$ ($n\geq 3$) with parallel mean curvature. Assume there is a distance function
$d$ on $M$, the metric ball $B_{4R_0}(x_0)$ given by $d$ satisfies DVP-condition for some $x_0\in M$ and $R_0\in (0,+\infty]$,
and the Gauss image of $B_{4R_0}(x_0)$ is contained in a compact set $K\subset \Bbb{W}_{1/3}$. Then there exists a constant $\de_1\in (0,1)$,
depending only on $K,K_1,K_2$, not depending on $x_0$ and $R_0$, such
that the image of $B_{\de_1R_0}(x_0)$ under the Gauss map is contained in
$\{S\in \grs{n}{m}:w(S,P)\geq \f{1}{12}+\f{3}{4}c\}$ for some $P\in \grs{n}{m}$, where $c:=\inf_K r$.
\end{lem}

\begin{lem}\label{sh2}
Let $M^n$ be
a minimal submanifold in $\R^{n+m}$ ($n\geq 3$)  with $rank(\g)\leq 2$
($\g$ denotes the Gauss map). Assume there is a distance function
$d$ on $M$, the metric ball $B_{4R_0}(x_0)$ given by $d$ satisfies the
DVP-condition for some $x_0\in M$ and $R_0\in (0,+\infty]$,
and the Gauss image of $B_{4R_0}(x_0)$ is contained in a compact set $K\subset \Bbb{W}_{0}$. Then there exists a constant $\de_1\in (0,1)$,
depending only on $K,K_1,K_2$, not depending on $x_0$ and $R_0$, such
that the image of $B_{\de_1R_0}(x_0)$ under the Gauss map is contained in
$\{S\in \grs{n}{m}:w(S,P)\geq \f{3}{4}c\}$ for some $P\in \grs{n}{m}$, where $c:=\inf_K r$.
\end{lem}

\begin{proof}
As above, we only prove Lemma \ref{sh1}.

Since $K\subset \Bbb{W}_{1/3}$ is compact, $c:=\inf_K r>\f{1}{3}$ and $\Th:=\big\{t\in (-\pi,\pi):t=\th(S)\text{ for some }S\in K\big\}$
is a compact subset of $(-\pi,\pi)$. Obviously $K\subset \overline{\Bbb{W}}_{c,\Th}$.

By Lemma \ref{target}, $\overline{\Bbb{W}}_{c,\Th}$ is diffeomorphic to $\Th\times \overline{\Bbb{D}}_{c^{-1}-1}^{nm-1}$. Thus $\overline{\Bbb{W}}_{c,\Th}$
can be seen as a bounded convex domain in $\R^{nm}$ equipped with the
induced metric. The eigenvalues of the metric matrices are bounded from
below by $K_3$ and from above by $K_4$, where $K_3,K_4$ are positive constants depending only on $K$.

The Ruh-Vilms Theorem \cite{r-v} implies that $\g$ is a harmonic function. Putting $u=\g$ in (\ref{es1}) gives $|du|^2=|B|^2$.
Let $\{H(\cdot,t):t\in \Th\}$ be a
family of smooth functions on $\overline{\Bbb{W}}_{c,\Th}$ as constructed in Proposition \ref{subhar4}. Choosing one of the functions
as a test function in (\ref{es1}) yields
\begin{equation}\label{es2}
H(\g(x),t)\leq H(\bar{\g}_R,t)+C_7C_9\Big(\f{R^2}{V(x_0,R)}\int_{B_{R}(x_0)}|B|^2*1\Big)^{\f{1}{2}}.
\end{equation}
for arbitrary $R\leq R_0$ and all $x\in B_{\f{R}{4}}$. Here
\begin{equation}
C_9:=\sup_\Th\sup_{\overline{\Bbb{W}}_{c,\Th}}\big|\n H(\cdot,t)\big|
\end{equation}
and the last term in (\ref{es1}) has been thrown out, since $H(\cdot,t)\circ \g$ is a subharmonic function.

For arbitrary $S\in \overline{\Bbb{W}}_{c,\Th}$, $w(S,P_{\th(S)})=r(S)\geq c>\f{3}{4}c+\f{1}{12}$, which implies
$H(S,\th(S))<1$ by Proposition \ref{subhar4}(ii). Since $\overline{\Bbb{W}}_{c,\Th}$ is compact, there exists
a positive constant $\ep_1$, such that
\begin{equation}\label{es3}
H(S,\th(S))\leq 1-\ep_1\qquad \text{for all }S\in \overline{\Bbb{W}}_{c,\Th}.
\end{equation}

By Lemma \ref{cur} (curvature estimates), if we denote
\begin{equation}
\de_1:=\f{1}{4}\exp(-C_6C_7^{2}C_9^{2}\ep_1^{-2})
\end{equation}
then there exists $R\in [4\de_1 R_0,R_0]$, such that
\begin{equation}\label{es4}
\f{R^2}{V(x_0,R)}\int_{B_R(x_0)}|B|^2*1\leq C_7^{-2}C_9^{-2}\ep_1^2.
\end{equation}

Letting $t=\th(\bar{\g}_R)$ and substituting (\ref{es3}) and (\ref{es4}) into (\ref{es2}) gives
\begin{equation}
H\big(\g(x),\th(\bar{\g}_R)\big)\leq H(\bar{\g}_R,\th(\bar{\g}_R))+C_7C_9(C_7^{-2}C_9^{-2}\ep_1^2)^{\f{1}{2}}\leq 1-\ep_1+\ep_1=1
\end{equation}
for all $x\in B_{\f{R}{4}}(x_0)$. Hence by Proposition \ref{subhar4}(ii),
$$\g\big(B_{\de_1R_0}(x_0)\big)\subset \g\big(B_{\f{R}{4}}(x_0)\big)\subset \big\{S\in \grs{n}{m}:w(S,P_{\th(\bar{\g}_R)})\geq \f{1}{12}+\f{3}{4}c\big\}.$$

\end{proof}

Given $\g\big(B_{\de_1R_0}(x_0)\big)\subset \big\{S\in \grs{n}{m}:w(S,P)\geq \f{1}{12}+\f{3}{4}c>\f{1}{3}\big\}$, one can start an iteration
as in \cite{j-x-y2} to get a-priori estimates for the Gauss image:
\begin{lem}
Let $M^n$ be a submanifold in $\R^{n+m}$ ($n\geq 3$) with parallel mean curvature. Assume there is a distance function
$d$ on $M$, the metric ball $B_{4R_0}(x_0)$ defined by $d$ satisfies
the DVP-condition for some $x_0\in M$ and $R_0\in (0,+\infty]$,
and the Gauss image of $B_{4R_0}(x_0)$ is contained in a compact set $K\subset \Bbb{W}_{1/3}$. Then for arbitrary $\ep>0$, there exists
a constant $\de_2\in (0,1)$, depending only on $K,K_1,K_2,\ep$, not depending on $x_0$ and $R_0$, such that
\begin{equation}
w(\g(x),\g(x_0))\geq 1-\ep\qquad \text{on }B_{\de_2 R_0}(x_0).
\end{equation}
In particular, if $M$ is a minimal submanifold with $rank(\g)\leq 2$,
then the condition on the Gauss image
can be relaxed to $\g(B_{4R_0}(x_0))\subset K\subset \Bbb{W}_0$.
\end{lem}

Letting $R_0\ra +\infty$ we arrive at a Bernstein type theorem.

\begin{thm}\label{Ber2}
Let $M^n$ be a submanifold in $\R^{n+m}$ ($n\geq 3$) with parallel mean curvature. Assume there is a distance function $d$
on $M$, such that $M$ satisfies the DVP-condition and the diameter of
$M$  with respect to $d$ is infinite; and there exist $P,Q\in \grs{n}{m}$
that are S-orthogonal to each other, such that $\big(w(\g(x),P),w(\g(x),Q)\big)$ is contained in a compact subset $K$ of
$\overline{\Bbb{D}}\backslash \big(\overline{\Bbb{D}}_{1/3}\cup \{(a,0):a\leq 0\}\big)$ for all $x\in M$. Then $M$ has to be an affine linear subspace.\newline
In particular, if $M$ is a minimal submanifold with $rank(\g)\leq 2$,
then the assumptions on the Gauss image
can be replaced by $\big(w(\g(x),P),w(\g(x),Q)\big)\in K\subset \overline{\Bbb{D}}\backslash \{(a,0):a\leq 0\}$.
\end{thm}

The above general Theorem can be applied to graphic submanifolds as follows.

\begin{thm}\label{Ber3}
Let $z^\a=f^\a(x^1,\cdots,x^n),\ \a=1,\cdots,m$, be smooth functions
defined everywhere in $\R^n$ ($n\geq 3,m\geq 2$), such that their graph $M=(x,f(x))$ is a
submanifold with parallel mean curvature in $\R^{n+m}$. Suppose that
there exist $\be_0<+\infty$ and $\be_1<3$, such that
\begin{equation}
\De_f:=\Big[\det\Big(\de_{ij}+\sum_\a \f{\p f^\a}{\p x^i}\f{\p f^\a}{\p x^j}\Big)\Big]^{\f{1}{2}}\leq \be_0.\label{be2}
\end{equation}
and
\begin{equation}\label{slope}
\De_f\leq \be_1\Big(1+\big(\f{\p f^1}{\p x^1}\big)^2\Big)^{\f{1}{2}}.
\end{equation}
Then $f^1,\cdots,f^m$ have to be affine linear (representing an affine $n$-plane).
\end{thm}

\begin{proof}
$F: \R^n\ra M$ defined by
$$x\mapsto (x,f(x))$$
is obviously a diffeomorphism. Thus $M$ can be viewed as an $n$-dimensional Euclidean space equipped with metric
$g=g_{ij}dx^idx^j$. Here
$$g_{ij}=\Big\lan F_* \f{\p}{\p x^i},F_*\f{\p}{\p x^j}\Big\ran =\Big\lan \ep_i+\f{\p f^\a}{\p x^i}\ep_{n+\a},\ep_j+\f{\p f^\a}{\p x^j}\ep_{n+\a}\Big\ran=\de_{ij}+\f{\p f^\a}{\p x^i}\f{\p f^\a}{\p x^j}$$
with $\{\ep_i,\ep_{n+\a}\}$ being the canonical orthonormal basis of $\R^{n+m}$. Let $Df:=(\f{\p f^\a}{\p x^i})$
be an $(n\times m)$-matrix valued function on $\R^n$, then
$$(g_{ij})=I_n+Df(Df)^T.$$
Hence $(g_{ij})\geq I_n$. In conjunction with $\det(g_{ij})=\De_f^2\leq \be_0^2$, we can deduce that all the eigenvalues
of $(g_{ij})$ take values between $1$ and $\be_0^2$. Denote $d:M\times M\ra \R$
$$d\big(F(x),F(y)\big)=|x-y|,$$
then the diameter of $M$ is $+\infty$ and it has been shown in Remark
\ref{DVP} that $M$ satisfies a DVP-condition.

Let $P,Q\in \grs{n}{m}$ whose Pl\"ucker coordinates are
$\ep_1\w\cdots\w \ep_n$ and $\ep_{n+1}\w \ep_2\w\cdots\w\ep_n$,  respectively.
Obviously,  $P,Q$ are S-orthogonal to each other. Denote by $\psi$ the Pl\"ucker embedding,
then
$$\psi\circ \g=\big(\ep_1+\pd{f^\a}{x^1}\ep_{n+\a}\big)\w\cdots\w \big(\ep_n+\pd{f^\a}{x^n}\ep_{n+\a}\big)$$
and by a direct computation,
\begin{eqnarray*}
w(\g,P)&=&\f{\lan \psi\circ\g, P\ran}{\lan \psi\circ\g,\psi\circ \g\ran^{\f{1}{2}}\lan P,P\ran^{\f{1}{2}}}=\De_f^{-1},\\
w(\g,Q)&=&\f{\lan \psi\circ\g, Q\ran}{\lan \psi\circ\g,\psi\circ \g\ran^{\f{1}{2}}\lan Q,Q\ran^{\f{1}{2}}}=\pd{f^1}{x^1}\De_f^{-1}.
\end{eqnarray*}
Hence $\big(w(\g,P),w(\g,Q)\big)\in K\subset \overline{\Bbb{D}}\backslash \big(\overline{\Bbb{D}}_{1/3}\cup \{(a,0):a\leq 0\}\big)$
if and only if
$$\be_1^{-2}\leq w(\g,P)^2+w(\g,Q)^2=\Big(1+\big(\pd{f^1}{x^1}\big)^2\Big)\De_f^{-2}$$
for a constant $\be_1<3$. This is equivalent to (\ref{slope}). So the Bernstein type result follows from Theorem \ref{Ber2}.

\end{proof}
\bigskip

\begin{rem}
Obviously $\be_1\Big(1+\big(\f{\p f^1}{\p x^1}\big)^2\Big)^{\f{1}{2}}\geq \be_1$ on the right hand side of (\ref{slope}). Hence Corollary
\ref{Ber3}  improves Theorem 6.2 in \cite{j-x-y2}. We note that $\pd{f^1}{x^1}$ in (\ref{slope}) can be replaced by
$\pd{f^\a}{x^i}$ for arbitrary $1\leq \a\leq m$ and $1\leq i\leq n$.
\end{rem}

\bigskip\bigskip

\Section{Appendix}{Appendix}
Let $\hh{}$ denote the quaternions with the standard basis $1, i, j, k$ and $\cc{}=\ir{}+\ir{}i, \hh{}=\cc{}+\cc{}j.$  Let
$q=z_1-\bar z_2 j\in\hh{}$ with $z_1, z_2\in\cc{}$. Then $\bar q=\bar z_1+\bar z_2 j$ and
$$\aligned
q i \bar q&= (z_1-\bar z_2 j) i (\bar z_1+ \bar z_2 j)\\
&=(\bar z_2 k+z_1 i)(\bar z_1+ \bar z_2 j)\\
&=\bar z_2 k \bar z_1 + z_1 i \bar z_1 +\bar z_2 k \bar z_2 j +z_1 i \bar z_2 j\\
&=z_1\bar z_2 k + |z_1|^2 i-|z_2|^2i + z_1\bar z_2k\\
&=(|z_1|^2-|z_2|^2)i+2 z_1\bar z_2 k
\endaligned$$
which coincides with the usual Hopf map
\begin{eqnarray*}\e&:& \hh{}\to \text{Im}\, \hh{}\\
\e&=& ((|z_1|^2-|z_2|^2,\;  2 z_1\bar z_2): \ir{4}\to \ir{3}
\end{eqnarray*}
with $\e(S^3)\in S^2$. Let
$$\zeta(x)=s(r)\e\left(\f{x}{r}\right)=s(r)r^{-2}\e(x):=\tilde s(r)\e(x)$$ with $r=|x|$.
It was shown in (\cite{h-l} Theorem 3.2, p. 135) that
$$x\to (x, \zeta(x))$$ define a
coassociative $4-$ submanifold in $\ir{7}$ invariant under $S^3$, provided
$$s(4s^2-5 r^2)^2=C,\qquad C\in\ir{}.$$ Those are area-minimazing
smooth minimal submanifolds except in the
case $C=0.$ When $C=0$ then $s(r)=\f{\sqrt{5}}{2}r$ and then $\tilde s(r)=\f{\sqrt{5}}{2r}$, the function $\zeta(x):\ir{4}\to \ir{3}$ given by
$$\zeta=\f{\sqrt{5}}{2r}\e(x)$$
defines a cone over the entire $\ir{4}$. This was discovered by Lawson and Osserman \cite{l-o}. This LO-cone shows that
 Moser's theorem that entire minimal graphs of bounded slope are
 affine linear cannot be extended to the case of dimension  $4$ and codimension $3$.
In this appendix, we  compute some important geometric quantities of this remarkable example.

Put
$$x=(z_1, z_2)\in \cc{2}=\ir{4},\quad z_1=r_1e^{i\th_1},\quad z_2=r_2e^{i\th_2},\quad |x|^2=r^2=r_1^2+r_2^2.$$
The flat metric on $\ir{4}$ reads
$$ds^2=dr_1^2+r_1^2d\th_1^2+dr_2^2+r_2^2d\th_2^2.$$
The orthonormal basis on $T_x\ir{4}$ is given by $\{e_0, e_1, e_2, e_3\},$ where
$$e_0=\pr{}=\f{r_1}{r}\pd{}{r_1}+\f{r_2}{r}\pd{}{r_2},\quad e_1=\f{r_2}{r}\pd{}{r_1}-\f{r_1}{r}\pd{}{r_2}$$
$$e_2=\f{r_2}{r_1r}\pd{}{\th_1}-\f{r_1}{r_2r}\pd{}{\th_2},\quad e_3=\f{1}{r}\left(\pd{}{\th_1}+\pd{}{\th_2}\right).$$
Now, we have
$$\e(x)=\left(r_1^2-r_2^2,\;  2r_1r_2 e^{i(\th_1-\th_2)}\right):\ir{4}\to \ir{3}.$$

Since
$$\pd{}{r_1}\tilde s(r)=\tilde s'\f{r_1}{r}, \quad \pd{}{r_2}\tilde s(r)=\tilde s'\f{r_2}{r},$$
then $$\aligned
\zeta_*e_0&=\left(\f{r_1}{r}\pd{}{r_1}+\f{r_2}{r}\pd{}{r_2}\right)(\tilde s\e)=\left(\tilde s'+\f{2\tilde s}{r}\right)\e,\\
\zeta_*e_1&=\f{r_2}{r}\pd{}{r_1}(\tilde s\e)-\f{r_1}{r}\pd{}{r_2}(\tilde s\e)=\f{\tilde s}{r}\left(4r_1r_2,\; 2(r_2^2-r_1^2)e^{i(\th_1-\th_2)}\right).
\endaligned$$
Since
$$\aligned
\zeta_*\pd{}{\th_1}&=\pd{}{\th_1}(\tilde s\e)=\tilde s\left(0,\; 2ir_1r_2e^{i(\th_1-\th_2)}\right),\\
\zeta_*\pd{}{\th_2}&=\pd{}{\th_2}(\tilde s\e)=\tilde s\left(0,\; - 2ir_1r_2e^{i(\th_1-\th_2)}\right),
\endaligned$$
then $$\aligned
\zeta_*e_2&=\f{r_2}{r_1r}\zeta_*\pd{}{\th_1}-\f{r_1}{r_2r}\zeta_*\pd{}{\th_2}=\tilde sr\left(0,\; 2i e^{i(\th_1-\th_2)}\right)\\
\zeta_*e_3&=0.
\endaligned$$

Put
$$
\aligned
\rho_0^2&=\f{1}{1+|\zeta_*e_0|^2}=\f{1}{1+(r^2\tilde s'+2r\tilde s)^2},\quad \rho_1^2=\f{1}{1+|\zeta_*e_1|^2}=\f{1}{1+4r^2\tilde s^2}\\
\rho_2^2&=\f{1}{1+|\zeta_*e_2|^2}=\f{1}{1+4r^2\tilde s^2},\quad \rho_3^2=\f{1}{1+|\zeta_*e_3|^2)}=1
\endaligned$$

Then the  Gauss map $\g$ for the coassociate $4-$submanifold is expressed by
$$(e_0, \zeta_*e_0)\w (e_1, \zeta_*e_1)\w (e_2, \zeta_*e_2)\w (e_3, \zeta_*e_3)$$
and the correspoinding $W-$matrix relative to $e_0\w e_1\w e_2\w e_3$ is

$$W=\left(\begin{array}{cccc}
            \rho_0 & 0  & 0 &0\\
                 0   & \rho_1 & 0&0 \\
                  0  & 0       & \rho_2 &0\\
                   0 & 0 &0  & 1
            \end{array}\right).$$
The nonzero Jordan angles are
$$\th_0=\arccos\rho_0,\; \th_1=\arccos\rho_1,  \; \th_2=\arccos\rho_2.$$

In particular, the Jordan angles of the image under the Gauss map for
the \textit{LO-cone}  can be obtained by substituting
$\tilde s=\f{\sqrt{5}}{2}r^{-1}$ in the above
expressions. Those are the following constants
$$\th_0=\arccos \f{2}{3},\quad  \th_1=\th_2=\arccos\f{\sqrt{6}}{6}.$$
The $w-$function is identically $\f{1}{9}$ and hence $v-$function,
denoted also by $\De_f$ for the graphic case, equals $9$. This fact
was originally verified by a computer program by Lawson-Osserman in \cite{l-o}. The LO-cone defined by $(f^1, f^2, f^3)$ on $\ir{4}$,
where
$$\aligned f^1&=\f{\sqrt{5}}{2}\f{(x^1)^2+(x^2)^2-(x^3)^2-(x^4)^2}{\sqrt{(x^1)^2+(x^2)^2+(x^3)^2+(x^4)^2}},\\
           f^2&=\f{\sqrt{5}}{2}\f{2(x^1x^3+x^2x^4)}{\sqrt{(x^1)^2+(x^2)^2+(x^3)^2+(x^4)^2}},\\
           f^3&=\f{\sqrt{5}}{2}\f{2(x^1x^3-x^2x^4)}{\sqrt{(x^1)^2+(x^2)^2+(x^3)^2+(x^4)^2}}.
           \endaligned$$
At $x^3\neq 0$ and $x^1=x^2=x^3=0,$ we have $\pd{f^2}{x^1}=\sqrt{5}$ and $(1+(\pd{f^2}{x^1})^2)^{\f{1}{2}}=\sqrt{6}.$

\bigskip\bigskip

\bibliographystyle{amsplain}

\end{document}